\documentclass[11pt]{article}
\usepackage{amssymb,amsthm,amsmath}
\usepackage{float}
\usepackage{amscd}
\usepackage{graphicx}
\usepackage{url}
\usepackage[total={7.0in,10.0in}, top=0.6in, left=0.7in, includefoot]{geometry}
\usepackage[all,cmtip]{xy}
\usepackage{xcolor}
\usepackage[pdftex,bookmarks, colorlinks, citecolor =blue, breaklinks]{hyperref}
\usepackage{hyperref}

\usepackage[all]{xy} 

\usepackage{enumitem}

\newcommand{\Eq}[1]{(\ref{eq:#1})}
\newcommand{\Th}[1]{Thm.~\ref{thm:#1}}
\newcommand{\Cor}[1]{Cor.~\ref{cor:#1}}
\newcommand{\Lem}[1]{Lem.~\ref{lem:#1}}

\newcommand{\Rem}[1]{Rem.~\ref{rem:#1}}

\newcommand{\Sec}[1]{\S \ref{sec:#1}}
\newcommand{\Fig}[1]{Fig.~\ref{fig:#1}}

\newcommand{\App}[1]{App.~\ref{app:#1}}

\newcommand{\Ex}[1]{Ex.~\ref{ex:#1}}
\newcommand{\Ass}[1]{Asm.~\ref{ass:#1}}

\newcommand{\InsertFig}[4]{
\begin{figure}[ht]
 \centerline{
 \includegraphics[width=#4]{./figs/#1}
 }
 \caption{{\footnotesize #2}
 \label{fig:#3}}
\end{figure}}

\newcommand{\InsertFigTwo}[5]{
\begin{figure}[ht]
 \centerline{
 \includegraphics[width=#5]{./figs/#1}
 \hskip 0.5in
 \includegraphics[width=#5]{./figs/#2}
 }
 \caption{{\footnotesize #3}
 \label{fig:#4}}
\end{figure}}

\newcommand{\bA}{{\mathbb{ A}}}

\newcommand{\bD}{{\mathbb{ D}}}
\newcommand{\bN}{{\mathbb{ N}}}
\newcommand{\bP}{{\mathbb{ P}}}
\newcommand{\bQ}{{\mathbb{ Q}}}
\newcommand{\bR}{{\mathbb{ R}}}
\newcommand{\bS}{{\mathbb{ S}}}

\newcommand{\bZ}{{\mathbb{ Z}}}

\newcommand{\cC}{\mathcal{C}}

\newcommand{\cE}{\mathcal{E}}

\newcommand{\cM}{\mathcal{M}}

\newcommand{\cS}{\mathcal{S}}
\newcommand{\cU}{\mathcal{U}}
\newcommand{\cV}{\mathcal{V}}
\newcommand{\cW}{\mathcal{W}}

\newcommand{\bu}{{\bf u}}
\newcommand{\bw}{{\bf w}}

\newcommand{\fg}{\mathfrak{g}}
\newcommand{\eps}{\varepsilon}
\newcommand{\vphi}{\varphi}

\newcommand{\am}{{\rm am}}
\newcommand{\sn}{{\rm {sn}}}
\newcommand{\cn}{{\rm cn}}
\newcommand{\dn}{{\rm dn}}
\newcommand{\cd}{{\rm cd}}

\newcommand{\interior}{{\rm Int}}

\newcommand{\diag}{{\rm diag}}

\newcommand{\cay}{{\rm Cay}}

\newcommand{\poris}[1]{\Delta_{#1}}
\newcommand{\cCi}{{\mathcal{C}}_i}
\newcommand{\cCo}{{\mathcal{C}}_o}
\newcommand{\pInc}{q_e}

\def\id{{\rm id}}
\def\gener{S}
\def\bnu{{\boldsymbol{\nu}}}
\def\bmu{{\boldsymbol{\mu}}}

\def\rhoin{\hat{\rho}}
\def\omegain{\hat{\omega}}

\newtheorem{thm}{Theorem} 
\newtheorem{lem}[thm]{Lemma}
\newtheorem{cor}[thm]{Corollary}

\newtheorem{defn}{Definition}

\theoremstyle{definition}
\newtheorem{remark}{Remark}
\newtheorem{example}[remark]{Example}

\newcommand{\bexam}[1][:]{\begin{example}[#1]}
\newcommand{\eexam}{\end{example}}

\newcommand{\beq}[1]{\begin{equation}\label{eq:#1}}
\newcommand{\eeq}{\end{equation}}

\newenvironment{se}[1]{\equation\label{eq:#1}\aligned}{\endaligned\endequation}
\newcommand{\bsplit}[1]{\begin{se}{#1}}
\newcommand{\esplit}{\end{se}}

\title{Symmetry Reduction and Rotation Numbers for Poncelet maps}

\author{
 H.E. Lomel\'{\i}$^{1}$ and J.~D.~Meiss$^{2}$\thanks
 {
 JDM was supported in part by NSF grants DMS-1812481 and CMMI-1537460.
 }
\smallskip\\
$^{1}$Department of Mathematics \\
 University of Texas at Austin\\
Austin, TX 78712 \\
	 {\tt lomeli@math.utexas.edu }
\smallskip\\
$^{2}$Department of Applied Mathematics\\
 University of Colorado \\
 Boulder, CO 80309-0526, USA\\
 {\tt James.Meiss@colorado.edu}
}

\date{\today}
\begin{document}
\maketitle

\begin{abstract}
\vspace*{1ex}
\noindent

Poncelet maps are circle maps constructed geometrically for a pair of nested ellipses;
they are related to the classic billiard map on an elliptical domain when the orbit has an
elliptical caustic. Here we show how the rotation number of the elliptical billiard map
can be obtained from a symmetry generated from the flow of a pendulum Hamiltonian system.
When such a symmetry flow has a global cross section, we previously showed that
there are coordinates in which the map takes a reduced, skew-product form on a covering space.
In particular, for elliptic billiard map this gives an explicit form for the rotation number
of each orbit.

We show that the family Poncelet maps on a pencil of ellipses is conjugate
to a corresponding family of billiard maps, and thus the Poncelet maps inherit
the one-parameter family of continuous symmetries.
Such a pencil has a single parameter, the \emph{pencil eccentricity}, which becomes
the modulus of the Jacobi elliptic functions used to construct a covering space that
simultaneously simplifies all of the Poncelet maps.
The rotation number of the Poncelet map for any element of a pencil can then be written
in terms of elliptic functions as well. An implication is that the rotation
number of the pencil has a monotonicity property: it is monotone increasing as the caustic ellipse shrinks.

The resulting expression for the rotation number gives an explicit condition for Poncelet porisms, the parameters for which
the rotation number is rational. For such parameters, an orbit of the corresponding Poncelet map is periodic: it
forms a polygon for \textit{any} initial point. These universal parameters also solve the inverse problem: given a
rotation number, which member of a pencil has a Poncelet map with that rotation number?
Explicit conditions are given for a general rotation numbers and we see how they are
related to Cayley's classic porism theorem.

%
%
\end{abstract}

\tableofcontents

\section{Introduction}\label{sec:Introduction}

Poncelet's celebrated 1822 theorem
\cite{Bos87,Fuchs07,Dragovic11,Dragovic14}
concerns polygons inscribed between a pair of ellipses: an ``outer'' ellipse, $\cCo$,
that encloses an ``inner" ellipse, $\cCi$. His theorem states that if there exists
an $n$-sided polygon with all vertices on $\cCo$ and all sides tangent to $\cCi$
then for each $z \in \cCo$ there is also another $n$-sided polygon with $z$ as a vertex that has the same properties.
Associated with the pair $(\cC_o,\cC_i)$ there is a circle diffeomorphism on $\cCo$ called the \emph{Poncelet map}
(or \emph{Poncelet traverse} \cite{Nohara12}).
Indeed, for \emph{any} pair of strictly convex curves $(\cC_o,\cC_i)$,
with $\cCi$ in the interior of $\cCo$, the homeomorphism
\beq{Pmap}
	P: \cCo \to \cCo,
\eeq
of the pair is defined so that for each $z_0\in\cCo$,
$z_1 = P(z_0)$ is the unique point on $\cCo$ such that the segment $\overrightarrow{z_0\,z_1\,}$
is tangent to $\cCi$ with positive orientation, i.e., so that $\cCi$ appears on the left when
moving from $z_0$ to $z_1$, see \Fig{PonceletMap}.
Clearly there also exists a second point $z_{-1}\in\cCo$
such that $\overrightarrow{z_{-1}\,z_0\,}$ is also tangent to $\cCi$;
this point corresponds to the inverse of the map: $z_{-1}=P^{-1}(z_0)$.

\InsertFig{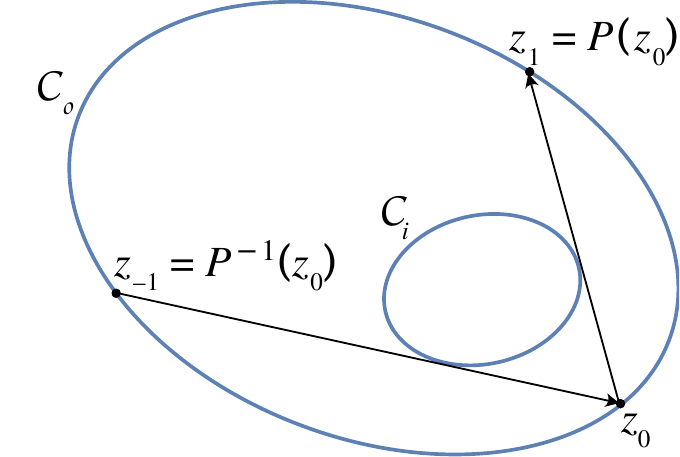}
{A Poncelet map $P: \cCo \to \cCo$ for a pair of nested, convex curves $\cCo$ and $\cCi$. }
{PonceletMap}{2.8 in}

Given an outer ellipse, the problem of characterizing the interior ellipses for which Poncelet
polygons exist is called the problem of \emph{Poncelet's porism}.
The poristic case then corresponds to the existence of a periodic orbit of $P$. In its simplest version, Poncelet's
theorem says that when $\cCo$ and $\cCi$ are ellipses, the existence of one
period-$n$ orbit implies that every orbit is period-$n$, i.e., that $P^n=\mathrm{id}$ on $\cCo$: this map is idempotent.

When the curves are concentric circles, the proof of Poncelet's theorem is simple, and
the family of polygons are simply obtained by rigid rotation
about the common center of the circles, a symmetry of the circle.
The proof of the general case has many versions, see e.g.
\cite{Griffiths77, Griffiths78, Kolodziej85, Cieslak10, Cima10, Kim10, Burski12, Cieslak16}.
In this paper we avoid using algebraic geometry (see e.g., \cite{Flatto09, Dragovic11} for this approach)
and concentrate on applying dynamical systems techniques.

Our first goal is to relate Poncelet's theorem to more general symmetries of integrable dynamical systems.
We will argue that the theorem follows from the existence of an equivariant,
continuous symmetry of $P$. In particular, this symmetry acts on a period-$n$ orbit of $P$ to deform
it into a continuous family of orbits with the same period.
Using this, we will address the fundamental problem of finding the rotation number of the Poncelet map
generated by a pair of concentric ellipses.

Even when the Poncelet map is not poristic, its dynamics are conjugate to rigid rotation.
In this paper, we find an explicit correspondence between the elements of a pencil of conics and rotation numbers.
For a given pencil, we obtain a family of Poncelet maps that are uniquely identified by the rotation number.
This correspondence can be written in terms of Jacobi elliptic functions that have a common modulus.
The modulus is, in fact, an invariant of the pencil that we will call the pencil eccentricity.

\subsection{Elliptic billiards}\label{sec:IntroBilliards}

We will build the symmetry of $P$ by conjugacy with the known symmetry of the elliptical billiard,
which we recall in \Sec{BilliardMaps}.
We start with the family of ellipses with eccentricities $0 < \eps < 1$ and foci at $(\pm1,0)$,
\beq{Ellipse}
 	E(\eps)=\left\{(x,y)\in\bR^2:\eps^2x^2+\left(\frac{\eps^2}{1-\eps^2}\right)y^2=1\right\}
	       = \left\{(x,y)\in\bR^2:\frac{x^2}{a^2}+\frac{y^2}{b^2}=1\right\},
\eeq
where $a = 1/\eps$ and $b = \sqrt{1-\eps^2}/\eps$ so that $0<b<a$ and $a^2-b^2=1$.
For a given eccentricity $\eps = f$, the billiard map, as studied by Birkhoff, is a symplectic twist map
defined on the cotangent bundle of $E(f)$, using the usual rule that the angle of incidence equals the angle of reflection.
In this case, the outer ellipse is $\cCo = E(f)$.
Orbits of the billiard map have caustics that are generically
confocal ellipses or hyperbolas, see e.g., \cite[Ch. 6]{Cornfeld82}. In \Sec{ConfocalPoncelet}
we recall the case that the caustic is an ellipse $\cCi = E(e)$ interior to $\cC_o$ with eccentricity $e$.
Note that generally,
\[
	E(e) \subset \interior(E(f)), \quad \mbox{whenever} \quad 0 < f < e < 1.
\]
This leads us to consider the triangle of eccentricity pairs
\beq{DeltaDefine}
	\Delta=\big\{(e,f)\in (0,1)\times(0,1): f<e\big\},
\eeq
as sketched in \Fig{ConstantRho}.
Whenever $(e,f) \in \Delta$, the orbits of the billiard map on $E(f)$ with caustic $E(e)$
induce a Poncelet map that we denote by
\[
	B_e^f : E(f) \to E(f).
\]

We showed in \cite{Dullin12} that equivariant symmetries of maps are related to the existence of covering
spaces when the symmetry flow has a global Poincar\'{e} section.
In this case, the map has a lift to the covering space that takes a skew-product form.
The billiard has an explicit symmetry that can be obtained from the flow of the pendulum,
and this provides explicit forms for the universal cover and the skew-product lift.

A consequence is that when the caustic is an ellipse, the billiard map $B_e^f$ restricted to each
symmetry orbit is conjugate to a rigid rotation with a rotation number that we denote
$\rhoin(e,f)$. In \Sec{ConfocalPoncelet},
we show that the symmetry implies that this rotation number can be expressed using elliptic functions:
\begin{thm}[Billiard Rotation Number]\label{thm:BilliardRotationNumber}
For each $(e,f)\in\Delta$, the rotation number $\rhoin(e,f)$ of $B_e^f$ is
\beq{BilliardRotationNumber}
	\rhoin(e,f)= \frac{F(\omegain(e,f),e)}{2K(e)}, \quad
	\omegain(e,f)=\arcsin\sqrt{\frac{e^2- f^2}{e^2(1-f^2)}},
\eeq
where $F(\omega,e)$ is the incomplete elliptic integral of the first kind, \Eq{EllipticIntegral},
and $K(e) = F(\tfrac{\pi}{2},e)$.
\end{thm}

\noindent
Versions of this standard result can also be found in \cite{Kolodziej85, Chang88, Ramirez14, Ramirez17, Kaloshin18}.

The connection between general Poncelet maps and elliptic billiards
allows us to find explicit expressions for the rotation number of the former as well.
While there has been much written about this connection, e.g.,
\cite{Chang88, Tabachnikov93, Dragovic98a, Dragovic06, Levi07, Dragovic10, Dragovic11, Garcia19},
as far as we know the equivariant billiard symmetry has not been previously used to analyze Poncelet maps.

\InsertFig{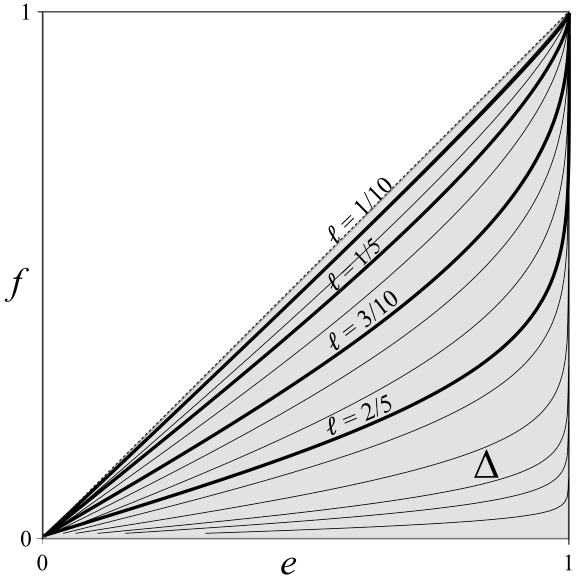}
{Level sets of the rotation number \Eq{BilliardRotationNumber} (curves) defined on the
set $\Delta$ \Eq{DeltaDefine} (gray).
The zero set of the Cayley polynomial, $\cay_{10}$ given by \Eq{Cayley10}, has four connected components, each corresponding
to a poristic curve $\Delta_\ell$ with $\rhoin(e,f) = \ell$ for $\ell=1/10$, $1/5$, $3/10$, and $2/5$.}
{ConstantRho}{3 in}

\subsection{Pencils of ellipses}

When $(\cCo,\cCi)$ are a nested pair of ellipses, but are not confocal, the Poncelet map as sketched in \Fig{PonceletMap}
is not the billiard map on $\cCo$: the usual equality of the angle of reflection and the angle of
incidence is abandoned. Nevertheless, $P$ is still a circle homeomorphism on $\cCo$, and thus,
as shown by Poincar\'e, every orbit has a rotation number.
In \Sec{EllipticPencils} we recall that a general ellipse $\cC$ is associated with
a $3\times 3$ symmetric matrix $C$ using projective coordinates so that
\beq{QuadraticForm}
 \cC = \left\{(x,y) \in \bR^2: \bu^T C \bu = 0, \, \bu = (x,y,1)^T \right\}.
\eeq
The assumptions on the matrix $C$ so that the associated curve $\cC$ is an ellipse are given in \Sec{Standing}.

Given a fixed pair of nested ellipses $\cC_1, \cC_2$ for which $\cC_2 \subset \interior(\cC_1)$
and their associated matrices $C_1$, $C_2$,
a \emph{pencil} is the one-parameter family generated by linear combinations
of the implicit equations:
\beq{PencilSet}
	\cC_{\bnu} = \left\{ (x,y) \in \bR^2 : \bu^T ( \nu_1C_1 + \nu_2 C_2) \bu = 0 \,,\, \bu = (x,y,1)^T \right\} ,
\eeq
where the parameters $\bnu = (\nu_1, \nu_2)$ determine the element of the pencil.

We show in \Sec{Diagonalization} that under natural assumptions the associated matrices $C_1$ and $C_2$
can simultaneously diagonalized by a congruency,
and we can choose this transformation so that the outer ellipse becomes the unit circle in the new coordinates.
Many of the properties of the Poncelet map $P: \cC_1 \to \cC_1$ with inner caustic $\cC_2$
are determined by the eigenvalues $\lambda_1$, $\lambda_2$, and $\lambda_3$ of the matrix $C_1^{-1}C_2$. We
show that these can be ordered so that
\[
	\lambda_1 > \lambda_2 > \lambda_3 > 0.
\]
For such eigenvalues, $\cC_\bnu$ is an ellipse interior to $\cC_1$ and satisfies the
standing assumptions of \Sec{Standing} precisely when the vector $\bnu = (\nu_1,\nu_1)$ is in the cone
\begin{equation}\label{eq:nuRestrictions}
 	\nu_2>0 \mbox{ and } \nu_1+\lambda_3\nu_2>0,
\end{equation}
as illustrated below in \Fig{PencilRegion}.

Given pencil eigenvalues as above, we define in \Sec{PencilE} the \emph{pencil eccentricity} for \Eq{PencilSet} as
\beq{PencilEccentricity}
	e(C_1,C_2) \equiv \sqrt{\frac{\lambda_1-\lambda_2}{\lambda_1-\lambda_3}}.
\eeq
We show in \Lem{ModifiedEccentricity} that this eccentricity is a pencil invariant: for a fixed outer ellipse,
\Eq{PencilEccentricity} is the same for any inner ellipse $\cC_\bnu$: $e(C_1,C_2) = e(C_1,C_\bnu)$.

Moreover, we show in \Sec{Projective} that the diagonalization
transformation can be thought of as a projective diffeomorphism that maps the general pencil \Eq{PencilSet} into a \emph{standard} pencil, see \Sec{StandardPencil}, for which the outer curve becomes the unit circle.

\InsertFig{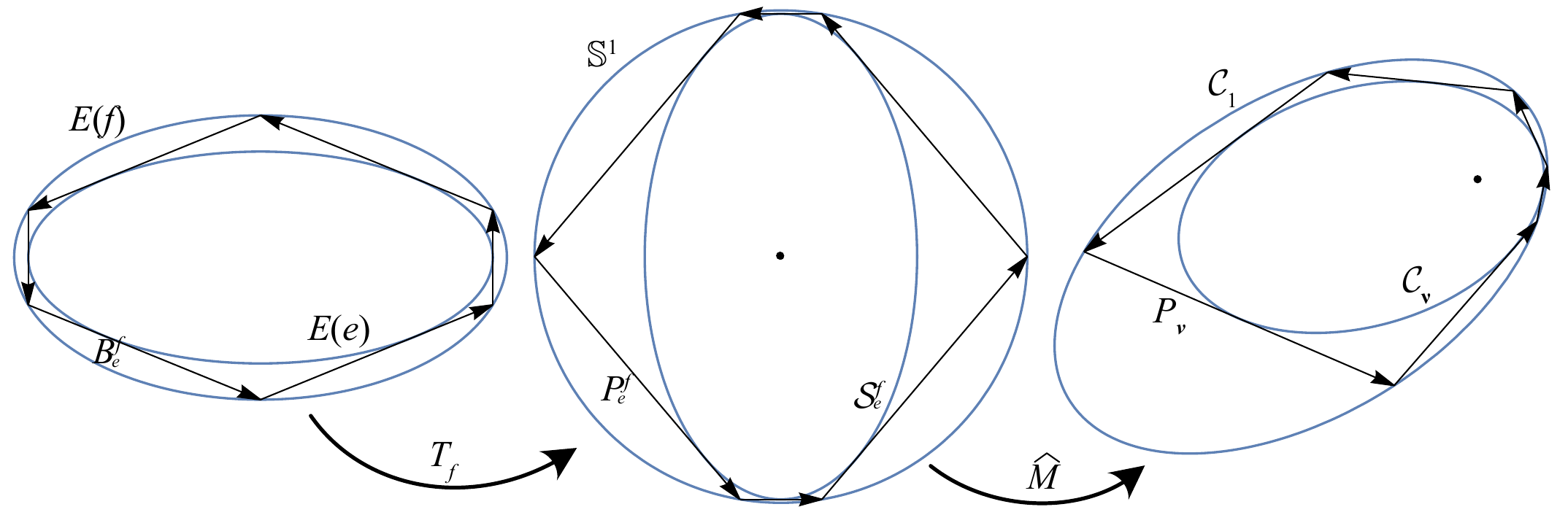}{Conjugacy between the billiard circle map $B_e^f$, the standard Poncelet map $P_e^f$  \Eq{PefDefine} for which the outer curve is a circle, and a general Poncelet map $P_\bnu$. Here the eccentricities are
$e = \sqrt{\tfrac{27}{32}} \approx 0.918559$ and $f=\tfrac{\sqrt{3}}{2} \approx 0.866025$ and the eigenvalues
are $(\lambda_1,\lambda_2,\lambda_3)=\left(\tfrac{1}{5},\tfrac{1}{8},\tfrac{1}{9}\right)$.}
{Conjugacy}{7in}

Fixing parameters $\bnu=(\nu_1,\nu_2)$ as in \Eq{nuRestrictions}, we
denote the Poncelet map for the outer ellipse $\cC_o = \cC_1$ and inner caustic $\cC_i = \cC_\bnu$  \Eq{PencilSet} by
\beq{PencilPoncelet}
	P_\bnu: \cC_1 \to \cC_1 .
\eeq
If $e$ is the pencil eccentricity \Eq{PencilEccentricity}, we can
construct a conjugacy between $P_\bnu$ and the billiard map $B_e^f$ for a given $f$
that is a function only of $e$, the eigenvalues, and $\bnu$ (see \Eq{Pencilf}).
A sketch of the conjugacy, \Fig{Conjugacy}, uses the standard pencil as an intermediary.
In \Sec{PonceletMaps} we show that this connects the rotation number \Eq{BilliardRotationNumber} of the billiard
map $B_e^f$ to that of more general Poncelet maps, leading---in \Sec{Parameterization}---to the following theorem.
\begin{thm}[Poncelet Rotation Number]\label{thm:PencilRotation}
If $\cC_\bnu$ is a pencil with eigenvalues $\lambda_1,\lambda_2,\lambda_3$ and eccentricity $e$ \Eq{PencilEccentricity},
then for each valid $\bnu$ \Eq{nuRestrictions} the rotation number of the corresponding Poncelet map \Eq{PencilPoncelet} is
\beq{PencilRotation}
	\rho(\bnu)= \frac{F(\omega(\bnu),e)}{2 K(e)},
	 \qquad
	\omega(\bnu)=\omega(\nu_1, \nu_2)
	 	=\arcsin\sqrt{\frac{\left(\lambda _1-\lambda _3\right) \nu _2}{\lambda _1 \nu _2+\nu _1}}.
\eeq
\end{thm}
\noindent
This relation allows the computation of the rotation number for a general Poncelet map for ellipses in a pencil
knowing only the eigenvalues of the matrix $C_1^{-1}C_2$. Note that the modulus of these elliptic functions
is fixed and equal to the pencil eccentricity and that $\rho_\bnu$ is a homogeneous function of the components of $\bnu$.

In \Sec{Monotonicity} we  prove that Poncelet maps and the rotation number \Eq{PencilRotation} satisfy a monotonicity property:
\begin{thm}[Monotonicity]\label{thm:Monotonicity}
Suppose $\cC_\bnu$ and $\cC_\bmu$ are ellipses in a pencil \Eq{PencilSet}
with parameters $\bnu$ and $\bmu$ satisfying \Eq{nuRestrictions} so that both are in the interior of an outer ellipse $\cC_1$.
If $\rho(\bmu)$ and $\rho(\bnu)$ are the rotation numbers \Eq{PencilRotation} of the corresponding
maps $P_\bmu$, and $P_\bnu$, then the following are equivalent:
\begin{enumerate}[nosep]
	\item $\cC_\bmu \subset \interior (\cC_\bnu)$,
	\item $\bnu \times \bmu > 0$,
	\item $\rho(\bmu) > \rho(\bnu)$.
\end{enumerate}
\end{thm}
\noindent
In particular, this shows that the rotation number is a monotonically increasing function as the inner caustic shrinks in the pencil.
In the limit when the inner ellipse becomes a point (e.g., $\nu_1+\lambda_3\nu_2 = 0$) and the rotation number becomes $\rho(\bnu) = \tfrac12$ as illustrated in \Fig{PencilRegion}.

\subsection{Poncelet Porisms and Cayley conditions}\label{sec:IntroPorism}

As we recall in \Sec{Poristic}, the Poncelet porism problem for the confocal case is equivalent
to finding a pair of eccentricities  $(e,f)\in\Delta$ for which $\rhoin(e,f) \in \bQ$.
Using the form \Eq{BilliardRotationNumber}, we find in  \Sec{ExplicitPorisms} an explicit
parameterization of the curve $\Delta_\ell \subset \Delta$
of parameters for which $\rhoin(e,f) = \ell \in[0,\tfrac12]$,
rational or irrational; the latter correspond to  quasi-periodic orbits of $B_e^f$.
Several examples can be seen in \Fig{ConstantRho}.

Subsets of $\Delta$ for which the orbit of a Poncelet map has period $N$ were found by Cayley,
see e.g. \cite{Griffiths78, Dragovic98a, Flatto09, Mirman10, Mirman12, Nohara12, Cieslak10, Garcia21}.
Though the Cayley conditions do not explicitly give the rotation number,
they are a useful first step in identifying poristic ellipses.
We argue below that the rotation number is a better tool.

For each $N>2$, Cayley found a polynomial, $\cay_N$,
with the property that the Poncelet map has a periodic orbit with period dividing $N$ if and only if $\cay_N =0$.
This method uses the expansion of the square root of the generalized characteristic polynomial,
\beq{CayleyExp}
	\sqrt{ \det(\lambda C_1 - C_2)} =  \alpha_0 + \alpha_1\lambda   + \alpha_2 \lambda^2 + \ldots ,
\eeq
where $C_1$ and $C_2$ are matrices representing the ellipses $\cC_1$ and $\cC_2$, respectively, and
we require that $\det(C_2)<0$.
For confocal ellipses, the coefficients $\alpha_j$ are
rational functions in the eccentricities $e$ and $f$ and the Cayley condition $\cay_N(\alpha_2,\ldots \alpha_{N-1}) = 0$
can be written as a polynomial in $(e,f)$.

Note that $\cay_N = 0$ implies that $\rhoin(e,f) = k/N$, for some $0<k<N/2$;
however, since the zero set necessarily includes each such rotation number
the equations become unmanageable very fast.
Finding the factors of the Cayley polynomial for each $\poris{k/N}$ is practical only for small values of $N$.
For example, for a pair of confocal ellipses the Cayley condition for $N=10$ has the form
\[\cay_{10}(\alpha_2,\ldots,\alpha_{9})=
\begin{pmatrix}
 \alpha_3 & \alpha_4 & \alpha_5 & \alpha_6 \\
 \alpha_4 & \alpha_5 & \alpha_6 & \alpha_7 \\
 \alpha_5 & \alpha_6 & \alpha_7 & \alpha_8 \\
 \alpha_6 & \alpha_7 & \alpha_8 & \alpha_9 \\
\end{pmatrix} = 0.
\]
The resulting set includes all the pairs $(e,f)\in\Delta$ for which the
billiard rotation number $\rhoin(e,f) = \tfrac{1}{10}$, $\tfrac15$, $\tfrac{3}{10}$, or $\tfrac25$, i.e.,
\beq{Cayley10}
(\cay_{10})^{-1}\{0\}\cap \Delta=
\Delta_{1/10}\cup\Delta_{1/5}\cup\Delta_{3/10}\cup\Delta_{2/5}.
\eeq
Thus the Cayley set splits into four components labelled by rotation number, as shown by the thicker curves in \Fig{ConstantRho}.
We recall several additional explicit examples in \Sec{PoristicParams} and \App{Porisms}.

By contrast, the explicit expression \Eq{BilliardRotationNumber} allows us to solve the inverse problem and directly
parameterize each set $\Delta_\ell$: given a rotation number $\ell$ and one of the eccentricities $e$ or $f$,
find the other eccentricity. In particular, we will show in \Lem{porism} that, for all $\ell\in(0,1/2)$, $\rhoin(e,f)=\ell$
if and only if $f=e\,\cd(2K(e)\ell,e)$, for the  Jacobi elliptic function $\cd$.
In \Sec{ExplicitRotation} we generalize this result to an arbitrary pencil \Eq{PencilSet}, giving
a complete solution of the inverse problem:
\begin{thm}[Inverse parameter conditions]\label{thm:PencilParam}
The Poncelet map \Eq{PencilPoncelet} with pencil eccentricity \Eq{PencilEccentricity}
has rotation number $\rho(\bnu) = \ell \in (0,\tfrac12)$ if and only if $(\nu_1,\nu_2)$ lie on the ray
\beq{InverseRotationMap}
 	\frac{\nu_1}{\nu_2}= \frac{\lambda_1\cn^2(2K(e)\ell,e)-\lambda_3,}{\sn^2(2K(e)\ell,e)}
\eeq
in the cone of parameters \Eq{nuRestrictions}.
\end{thm}

\noindent
The remarkable feature of this result is that we can determine the element of the pencil
whose Poncelet map has a given rotation number using only the eigenvalues of $C_1^{-1}C_2$:
the explicit conjugacy to the billiard map is not needed.

\subsection{Standard covering map and Full Poncelet Theorem}\label{sec:IntroStdCovering}

As another consequence of the results of \Sec{EllipticPencils}, we will find a covering
space that simultaneously simplifies all the Poncelet maps in \Sec{PonceletMaps}. Let
$\cC_\bnu$ be a pencil generated by two matrices $C_1,C_2$ that satisfy the
standing assumptions  \ref{ass:SA1}-\ref{ass:SA3}.
In \Sec{Parameterization}, we show that there exists a covering map
$\Pi_M:\bR\to\cC_1$ such that each Poncelet map $P_\bnu:\cC_1\to\cC_1$ has  a lift
of the form $B_\bnu(\theta)=\theta+\rho(\bnu)$, where $\rho(\bnu)$ is
the rotation number \Eq{PencilRotation}. This is shown in the following commutative diagram.
\[
	\xymatrix @R=3pc @C=3pc {
	{\bR\ }\ar[r]^-{B_\bnu} \ar[d]_{\Pi_M} &
 		\ \bR \ar[d]^{\Pi_M} \\
 		\cC_1\ \ar[r]^-{P_\bnu}& \ \cC_1\\ }
\]
We will see that the covering map $\Pi_M$ depends only on the eccentricity of the pencil and the matrix
that simultaneously diagonalizes $C_1$ and $C_2$.

In \Sec{GeneralPoncelet} we generalize these results to a Poncelet map \Eq{Pmap} with a sequence of inner ellipses
$\cC_\bnu$ in a given pencil. We will see that the corresponding Poncelet maps
commute. Using this, we obtain a more general version of
Poncelet's theorem with multiple inner ellipses.

\begin{thm}\label{thm:GeneralPoncelet}
Suppose that $\bnu^1,\bnu^2,\ldots,\bnu^N$ is a finite sequence of parameters
satisfying \Eq{nuRestrictions}, and that
$\cC_{\bnu^1}, \ldots, \cC_{\bnu^N}$ are the corresponding
inner ellipses in the  pencil.
Then the circle map on  $\cC_1$ defined by the composition
\beq{Composition}
	P_* = (P_{\bnu^N})^{k_N}\circ (P_{\bnu^{N-1}})^{k_{N-1}}\circ \cdots \circ (P_{\bnu^1})^{k_1},
\eeq
for $k_1,k_2,\ldots,k_N\in\bZ$, has the rotation number
\[
	\rho(P_*)=k_1\,\rho(\bnu^1)+k_2\,\rho(\bnu^2)+\cdots+k_N\,\rho(\bnu^N),
\]
where $\rho(\bnu)$ is given by \Eq{PencilRotation}.
In particular, if the rotation number of \Eq{Composition} is an integer, $\rho(P_*)\in\bZ$, then $P_*=id$. Similarly, if $P_*$  has a fixed point then $\rho(P_*)\in\bZ$ and hence  $P_*=id$.
\end{thm}

%

\section{Confocal Ellipses and Billiard Maps}\label{sec:BilliardMaps}
We start with the simplest case: the Poncelet map $P$ for a pair of confocal ellipses.
In this case, we will see that $P$ corresponds to an orbit of the billiard map with boundary
$\cCo = E(f)$ that has the inner ellipse $\cCi = E(e)$ as a caustic with $0<f<e<1$.
We recall in \Sec{TwistMaps} the twist map formulation for billiard maps, and then,
in \Sec{EllipticBilliards}, specialize to the integrable, elliptic case.
The elliptic billiard has a symmetry that is a Hamiltonian flow and allows us to apply a theorem from \cite{Dullin12}.
This theorem shows how to lift a map with a symmetry to a covering space on which it has a semi-direct product form.
These results will be applied to the confocal Poncelet map in \Sec{ConfocalPoncelet}.

\subsection{Billiards as Twist Maps}\label{sec:TwistMaps}

Classical billiard twist-map dynamics on smooth, convex curve $\cC$ gives rise to a twist map of an annulus.
To recall this, let $\gamma:\bR\to\bR^2$ be a $2\pi$-periodic,\footnote
{The twist map structure does not depend on having $2\pi$-periodicity, but we can assume
this without loss of generality. As an alternative one could use the arc length to parameterize the curve \cite{Meiss92a}.}
parametric representation of $\cC$:
\beq{ConvexCurve}
	\cC = \{\gamma(\vphi) \in\bR^2: \vphi \in \bR/2\pi\bZ\} .
\eeq
If $\gener(\vphi_0,\vphi_1)=\|\gamma(\vphi_0)-\gamma(\vphi_1) \|$ denotes the Euclidean length of
a line segment from $\gamma(\vphi_0)$ to $\gamma(\vphi_1)$, then
an orbit of the billiard consists of a sequence $\{\ldots, \vphi_{k}, \vphi_{k+1},\ldots \}$
so that the formal orbit length,
\[
	W(\ldots, \vphi_0,\vphi_1,\vphi_2, \ldots) = \sum_{k\in\bZ} \gener(\vphi_{k-1},\vphi_k),
\]
is stationary with respect to variations of $\vphi_{k}$. This is,
for all $k\in\bZ$,
\[
\frac{\partial \gener}{\partial \vphi_1}(\vphi_{k-1},\vphi_k)+
\frac{\partial \gener}{\partial \vphi_0}(\vphi_k,\vphi_{k+1})=0.
\]

Equivalently,
one can define a canonically conjugate ``momentum'' variable $r$, so that
the two definitions,
\begin{align*}
	r_0(\vphi_0,\vphi_1) &=-\frac{\partial \gener}{\partial \vphi_0}(\vphi_0,\vphi_1)
	 = \frac{\langle\gamma(\vphi_1)-\gamma(\vphi_0),\gamma'(\vphi_0)\rangle}{\gener(\vphi_0,\vphi_1)} ,\\
	r_1(\vphi_0,\vphi_1) &=\frac{\partial \gener}{\partial \vphi_1}(\vphi_0,\vphi_1)
	 =\frac{\langle \gamma(\vphi_1)-\gamma(\vphi_0),\gamma'(\vphi_1) \rangle}{\gener(\vphi_0,\vphi_1)} ,
\end{align*}
agree along an orbit: $r_0(\vphi_{k},\vphi_{k+1}) = r_1(\vphi_{k-1},\vphi_{k})$.
These definitions imply the standard rule that the ``angle of incidence'' is equal to the ``angle of reflection''.

Letting $\cM=\{(\vphi_0,\vphi_1)\in\bR/2\pi \bZ \times \bR/2\pi\bZ:\vphi_0\neq \vphi_1\}$,
the Lagrangian form of the billiard map $L: \cM \to \cM$, is
\[
	 (\vphi_{k},\vphi_{k+1}) = L(\vphi_{k-1}, \vphi_k).
\]
Alternatively, the canonical form of this map is defined on the open annulus
\beq{AnnulusDef}
	\bA= \big\{(\vphi,r)\in \bR/2\pi\bZ\times \bR: r^2<\|\gamma'(\vphi)\|^2\big\}.
\eeq
To obtain this, first define two functions $\kappa_{0,1}:\cM\to \bA$ by
\bsplit{kappaDefine}
	\kappa_0(\vphi_0,\vphi_1) &=&\left(\vphi_0, r_0(\vphi_0,\vphi_1)\right) , \\
	\kappa_1(\vphi_0,\vphi_1) &=&\left(\vphi_1, r_1(\vphi_0,\vphi_1) \right) .
\esplit
When $\cC$ is strictly convex and $\gamma$ is $C^2$, these can be shown to be diffeomorphisms \cite{Lomeli96}.
The canonical billiard map $B:\bA\to \bA$ is then given by $B=\kappa_1\circ\kappa_0^{-1}$,
so that
\beq{BilliardMap}
	(\vphi_{k+1},r_{k+1}) = B(\vphi_{k},r_{k}).
\eeq
Note that $L=\kappa_1^{-1}\circ \kappa_1^{-1} = \kappa_0^{-1}\circ B \circ \kappa_0$, so that $L$ is conjugate to $B$.
This construction implies that $B$ is symplectic with respect to
the two-form $\Omega=d\vphi\wedge dr$. Indeed, $B$ is the
exact symplectic twist map generated by $\gener(\vphi_0,\vphi_1)$ \cite{Meiss92a, Lomeli96}.

\subsection{Symmetry Reduction of Elliptic Billiards}\label{sec:EllipticBilliards}

Suppose that the billiard has boundary $\cC = E(f)$, \Eq{Ellipse}, for some $0 < f < 1$. Following \Eq{ConvexCurve},
the ellipse has the convenient parameterization
\beq{OuterEllipse}
	E(f) = \left\{\gamma(\vphi)=\frac{1}{f}(\cos\vphi, \sqrt{1-f^2}\,\sin \vphi)\in\bR^2: 0 \le \vphi< 2\pi \right\}.
\eeq
Noting that
$
\|\gamma'(\vphi)\|^2= 
					 \frac{1}{f^2}(1-f^2\cos^2 \vphi)
$,
then the annulus \Eq{AnnulusDef} becomes
\[
	\bA= \bigg\{(\vphi,r)\in \left(\bR/2\pi\bZ\right)\times \bR: \, H(\vphi,r) <\tfrac{1}{2}f^{-2} \bigg\} \;,
\]
where
\beq{Hamiltonian}
	H(\vphi,r)=\tfrac12r^2+\tfrac12\cos^2\vphi.
\eeq

As is well known, \cite{Lomeli96}, $H$ is an invariant under the elliptic billiard dynamics: $B^*H = H$;
thus the orbits of \Eq{BilliardMap} lie on contours of $H$, as shown in \Fig{PhasePlane}.
The associated symmetry is the Hamiltonian vector field $X_H$ generated
by $H$ with respect to $\Omega$, i.e., the function that satisfies $i_{X_H} \Omega = dH$.
To see this, let $\Phi:\bR\times \bA\to\bA$ be the (complete) flow of $X_H$.
Writing $\Phi_t=\Phi(t,\cdot)$, as usual, then
\[
 \frac{\partial}{\partial t}\Phi_t(\xi) = X_H(\Phi_t(\xi)), \qquad
 X_H(\Phi_t) = (H_r(\Phi_t),-H_\vphi(\Phi_t)), \qquad \Phi_0(\xi)=\xi.
\]

\begin{lem}
 $\Phi_t$ is an equivariant symmetry for the map $B$: $B\circ\Phi_t=\Phi_t \circ B$,
for all $t \in \bR$.
\end{lem}

\begin{proof}
Since $B$ is symplectic, then $B^*\Omega=\Omega$. Since $H = B^* H$, then
\[
	i_{X_H}\Omega = dH = d(B^* H) = B^*(i_{X_H} \Omega ) = i_{B^* X_H} \Omega,
\]
so $B^*X_H=X_H$. This implies that the flow $\Phi_t$ commutes with the map $B$.
\end{proof}

\InsertFig{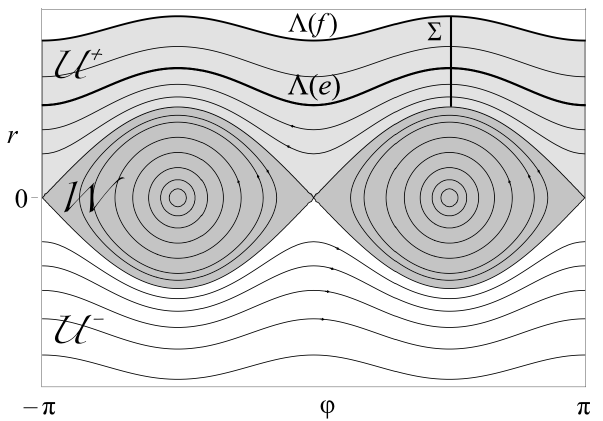}
{Contours \Eq{Hamiltonian} of the invariant $H$ in the annulus $\bA$, showing the sets $\cU^\pm$,
$\cW$ and its boundary $\cV$. The Poincar\'e section
$\Sigma = \{\vphi = \pi/2\} \cap \cU^+$ is indicated by the vertical segment. The invariant sets $\Lambda(f)$ and $\Lambda(e)$ are labeled.}
{PhasePlane}{3.5 in}

The phase space $\bA$ can be divided into four invariant subsets as shown in \Fig{PhasePlane}:
	$\cV =\{H =\tfrac{1}{2} \}$,
	$\cW =\{0 \le H < \tfrac{1}{2} \}$,
$\cU^+ =\{\tfrac{1}{2} <H < \tfrac{1}{2}f^{-2}, r>0 \}$, and
	$\cU^- =\{\tfrac{1}{2} <H < \tfrac{1}{2}f^{-2}, r<0 \}$.

The set $\cW$ contains the period two orbit $(-\tfrac{\pi}{2},0) \to (\tfrac{\pi}{2},0)$; more generally,
its two components are mapped from one to the other under $B$.
However, note that $\Phi$ leaves each component of $\cW$ invariant.
The boundary set $\cV$ separates $\cW$ from the outside regions $\cU^{\pm}$; it consists
of orbits that are homoclinic to the period-two saddle $(0,0) \to (\pi,0)$.

In this paper, we will consider only the dynamics of $B$ restricted to $\cU^+$ where the orbits encircle
the ellipse in a counter-clockwise direction. As is well-known, the orbits
of the elliptic billiard in $\cU^+$ have caustics that are ellipses \cite{Chang88, Lomeli96, Dragovic10}.
Each level set of $H$ in $\cU^+$ is a topological circle that we will denote by
\beq{LambdaDef}
	\Lambda(e)=\left\{(\vphi,r)\in\cU^+:H(\vphi,r)=\tfrac12 e^{-2}\right\},
\eeq
where $e$ is the eccentricity of the ellipse $E(e)$ that is the
corresponding caustic.

Thus, if $(\vphi_0,r_0)\in \Lambda(e)$, then the bi-infinite orbit
 $(\vphi_k,r_k)=B^k(\vphi_0,r_0)$ has the property that each segment
 $\overrightarrow{\,\gamma(\vphi_k)\,\gamma(\vphi_{k+1})\,}$ is tangent to the caustic $E(e)$.
Then $\cU^+$ is foliated by these level sets:
$\cU^+=\bigcup\{\Lambda(e):f<e<1\}$. 
The dynamics on $\cU^-$ are similar under a reflection. 

As we will see below, Poncelet's theorem can be viewed as a consequence
of a lifting theorem for equivariant maps that we proved in \cite{Dullin12}. To apply this theorem,
let
\beq{SigmaDefine}
	\Sigma=\left\{\left(\tfrac{\pi}{2},r\right)\in\bR^2:1<r<\tfrac{1}{f}\right\} = \left\{\vphi = \tfrac{\pi}{2}\right\} \cap \cU^+,
\eeq
be a (complete) Poincar\'e section for the flow $\Phi_t$ on the invariant set $\cU^+$, as indicated in \Fig{PhasePlane}.

\begin{thm}[Equivariant Lifting]\label{thm:Equivariant}
The map $p_0:\bR \times \Sigma \to \cU^+$ defined by $p_0(t,\xi)=\Phi_t(\xi)$
is a covering map, with a cyclic group of deck transformations
generated by $\psi(t,\xi)=(t+ T(\xi),\xi)$, where $T(\xi)$ is the period of
$\Phi_t(\xi)$. In addition, the lift of the billiard map $B|_{\cU^+}$,
$\widetilde{B}:\bR \times \Sigma \to \bR \times \Sigma$, is
\beq{SymmetryReduced}
	\widetilde{B}(t,\xi)=(t+\hat{\rho}(\xi),\xi),
\eeq
where $\hat{\rho}:\Sigma\to\bR$ is a smooth function on $\Sigma$.
\end{thm}

\begin{proof}
The manifold $\Sigma$ is simply connected and relatively closed in $\cU^+$.
Previously \cite{Dullin12},
we argued that such conditions imply that $p_0$ is a covering map, with deck transformations
generated by 
$\psi(\xi,\theta)=(\Phi_{-T(\xi)}(\xi),\theta+ T(\xi))$,
where $T$ is the smallest positive time such that
$\Phi_{T(\xi)}(\xi)\in\Sigma$. Since the flow $\Phi_t(\xi)$ is periodic, $T(\xi)$ is the period of
$\Phi_t(\xi)$, and $\Phi_{-T(\xi)}(\xi)=\xi$.

It was also proven in \cite{Dullin12} that,
under the covering space $p_0$, the lift $\widetilde B$ of
any diffeomorphism $B$ that is equivariant with respect to $\Phi_t$ has
to have the skew product form $\widetilde B(t,e)=(t+\hat{\rho}(\xi),g(\xi))$,
where $g:\Sigma\to\Sigma$ is a diffeomorphism of $\Sigma$. In this case,
the invariance of each level set implies that $g(\xi)=\xi$, giving \Eq{SymmetryReduced}.
\end{proof}

\subsection{Poncelet Maps for Confocal Ellipses}\label{sec:ConfocalPoncelet}

As in \Sec{IntroBilliards}, we denote the Poncelet map for the confocal
ellipses by by $B_e^f$, to indicate that it is conjugate to the
billiard map restricted to the invariant set $\Lambda(e)$, defined in \Eq{LambdaDef}.


\begin{lem}[Confocal Poncelet Map]\label{lem:BefMap}
Given eccentricities $0<f<e<1$, let
\(
	B_e^f:E(f)\to E(f)
\)
denote the Poncelet map with respect to the confocal
ellipses $E(f)$ and $E(e)$. Then $B_e^f$ is conjugate to the elliptical billiard
map $B|_{\Lambda(e)}$, i.e., the billiard map for the boundary $E(f)$ restricted
to orbits with positive orientation that have $E(e)$ as a caustic.
\end{lem}

\InsertFig{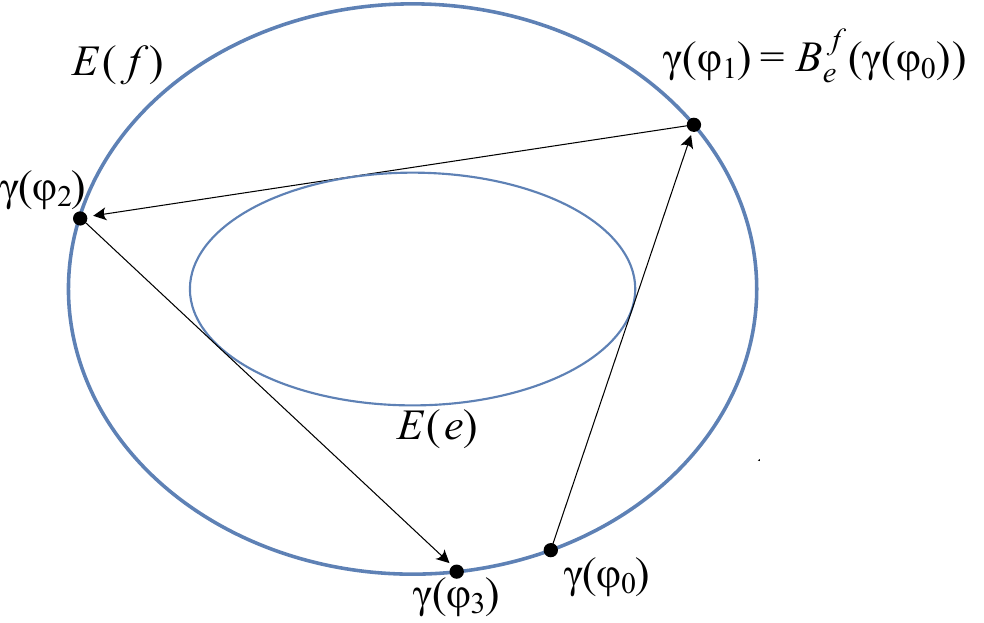}
{Confocal Poncelet map $B_e^f$ on $E(f)$ with caustic $E(e)$.}
{confocal}{3 in}

\begin{proof}
Fixing $0 <f < 1$, we parameterize $E(f)$ using $\gamma: \bS^1 \to \bR^2$ as in \Eq{OuterEllipse}. As sketched in \Fig{confocal}, if $(\vphi_k,r_k)=B^k(\vphi_0,r_0)$ is the billiard orbit
for $(\vphi_k,r_k) \in \Lambda(e)$, then each point $z_k = \gamma(\vphi_k)\in E(f)$,
and each oriented segment $\overrightarrow{ z_k\,z_{k+1}\,}$
is tangent to $E(e)$ with positive orientation. This implies that
\(
	B_e^f(\gamma(\vphi_k))= z_{k+1} = \gamma(\vphi_{k+1}),
\) for all $ k \in \bZ$.
Thus the Poncelet map $B_e^f$ is conjugate to the elliptic billiard map $B|_{\Lambda(e)}$.
\end{proof}

It is this link between the elliptical billiard and Poncelet maps that we will use in the following sections to study general Poncelet maps.

\begin{remark}
To define several of the covering spaces that will reveal the structure of Poncelet maps,
we use the Jacobi elliptic functions $\sn$, $\cn$, $\dn$ and $\am$ with modulus $e$. These have period $4K(e)$,
where $K(e) = F(\tfrac{\pi}{2}, e) $ is the complete elliptic integral of the first kind, and
\beq{EllipticIntegral}
	F(\phi,e) \equiv \int_0^\phi\frac{d\,\tau}{\sqrt{1-e^2\sin^2\tau}}.
\eeq
\end{remark}

\begin{thm}[Confocal Covering Map]\label{thm:ConfocalCover}
For $0<f<e<1$, the map $\pi_e^f:\bR\to E(f)$ defined by
\beq{maincover}
     \pi_e^f(\theta)=\frac{1}{f}
     \begin{pmatrix}
         -\sn\left(4K(e)\theta,e\right) \\
         \sqrt{1-f^2}\,\cn\left(4K(e)\theta,e\right) \\
     \end{pmatrix} \;,
 \eeq
is a covering map of $E(f)$, with group of deck transformations
generated by $\theta\mapsto \theta+1$.
Moreover, the Poncelet map $B_e^f$ has a lift of the form
\beq{LiftedBilliard}
	\widetilde B_e^f(\theta)=\theta+\rhoin(e,f),
\eeq
where $\rhoin(e,f)$ only depends on $(e,f)$. In addition,
we can assume that $0<\rhoin(e,f)<1$.
\end{thm}

\begin{proof}
An explicit formula for the lift $p_0: \bR \times \Sigma \to \cU^+$ of \Th{Equivariant} can be obtained
using properties of Jacobi elliptic functions.
Letting $\xi=(\tfrac{\pi}{2} , r) \in \Sigma$, \Eq{SigmaDefine}, then the explicit
form of the flow of the pendulum, see e.g., \cite{Ochs11}, implies that
\begin{equation}
 p_0(t,\xi)= \Phi_t (\xi ) =
 \begin{pmatrix}
 	\displaystyle \am\left(r t , r^{-1}\right)+{\pi}/{2} \\
 	\displaystyle r\dn\left(r t , r^{-1}\right) \\
 \end{pmatrix}.
\end{equation}
To specialize this to the level set $\Lambda(e) \subset \cU^+$, we let
\(
	\xi_e = (\pi/2,e^{-1}) = \Sigma \cap \Lambda(e),
\)
and $\pInc:\bR \to\Lambda(e)$ be the function
\beq{covering}
 	\pInc(\theta) \equiv p_0(4 eK(e)\theta,\xi_e) =
 \begin{pmatrix}
 \displaystyle \am\left(4K(e) \,\theta, e\right)+{\pi}/{2} \\
 \displaystyle e^{-1}\,\dn\left(4K(e) \,\theta, e\right) \\
 \end{pmatrix}.
\eeq
We define the inclusion $i:\Lambda(e)\hookrightarrow \cU^+$, and
the embedding $j: \bR \to \bR\times\Sigma$ given by
$j(\theta)= (4e\,K(e) \,\theta,\xi_e)$.
The implication is that the range of $p_0\circ j$ is $\Lambda(e)$, and
that the following diagram commutes:
\[
\xymatrix @R=3pc @C=3pc {
 {\bR\ }\ar[r]^-{j} \ar[d]_{\pInc} &
 \ \bR\times\Sigma \ar[d]^{p_0} \\
 \Lambda(e)\ \ar@{^{(}->}[r]^-{i}& \ \cU^+ \\
}
\]
In other words, $i\circ\pInc=p_0\circ j$.
This implies that $\pInc$ is a covering map of $\Lambda(e)$.
A simple computation shows that $B|_{\Lambda(e)}$, the billiard map restricted to the level set $\Lambda(e)$, obeys
\beq{qLift}
	\left.B\right|_{\Lambda(e)}\circ\pInc(\theta)= \pInc(\theta+\rhoin(e,f)),
\eeq
for all $\theta$, where $\rhoin(e,f)=\hat{\rho}(\xi_e)/(4K(e))$.
This shows that $\widetilde B_e^f(\theta)=\theta+\rhoin(e,f)$
is a lift of $\left.B\right|_{\Lambda(e)}$.

We let $J_f:\Lambda(e)\to E(f)$ be the diffeomorphism
\(
	J_f(\vphi,r)=\frac{1}{f}\big(\cos\vphi,\sqrt{1-f^2}\sin\vphi\big).
\)
Then the composition $J_f\circ \pInc$ is a covering map of $E(f)$, and
$\pi_e^f=J_f\circ \pInc$. Note that the circle map $B_e^f$ satisfies
\(
	J_f\circ \left.B\right|_{\Lambda(e)}=B_e^f\circ J_f.
\)
The implication is that the diagram
\[
\xymatrix @R=3pc @C=3pc {
 {\bR }\ar[r]^{{\widetilde B}_e^f} \ar[d]_{\pInc}\ar@/_2.5pc/[dd]_{\pi_e^f} &
 \ar@/^2.5pc/[dd]^{\pi_e^f} {\bR }\ar[d]^{\pInc}\\
 \Lambda(e)\ar[r]^{\left.B\right|_{\Lambda(e)}}\ar[d]_{J_f} & \Lambda(e)\ar[d]^{J_f}\\
 E(f) \ar[r]^{B_e^f}& E(f)
}
\]
commutes. This also implies that the map $\widetilde B_e^f$, given by \Eq{LiftedBilliard},
is a lift of $B_e^f$ to the covering map $\pi_e^f$.
Given that the group of deck transformations is generated by $\theta\mapsto \theta+1$,
we can assume that $0<\rhoin(e,f)<1$.
\end{proof}

By \Eq{LiftedBilliard}, the rotation number of $B_e^f$ is $\rhoin(e,f)$.
We can now prove that this is given by \Eq{BilliardRotationNumber}, the result mentioned in \Sec{Introduction}:

\begin{proof}[Proof of \Th{BilliardRotationNumber} (Billiard Rotation Number)]
Suppose $(e,f)\in \Delta$ and fix $\omega= \omegain(e,f)$ as in \Eq{BilliardRotationNumber}.
Since $0<f<e<1$, $0<\omega<\pi/2$. Using the annular coordinates \Eq{AnnulusDef}, we define
\beq{twoangles}
 \varphi_0 = \tfrac{\pi}{2}-\omega, \qquad
 \varphi_1 = \tfrac{\pi}{2}+\omega,\qquad
 r_0 = e^{-1}\sqrt{1 -e^2\sin^2(\omega)}.
\eeq
Since
\(
	H(\varphi_0,r_0)=H(\varphi_1,r_0)=\tfrac{1}{2}e^{-2},
\)
and $r_0>0$, it is clear that $(\varphi_0,r_0),(\varphi_1,r_0)\in\Lambda(e)$.
Moreover, using \Eq{kappaDefine} we have that
$\kappa_0(\varphi_0, \varphi_1) =(\varphi_0, r_0) $ and
$\kappa_1(\varphi_0, \varphi_1) =(\varphi_1, r_0)$.
The implication is that $B|_{\Lambda(e)}(\varphi_0, r_0)=(\varphi_1, r_0)$.

Letting
\beq{theta12Def}
	\theta_1 = -\theta_0=\frac{F(\omega,e)}{4K(e)},
\eeq
then from \Eq{covering}, usual identities for Jacobi elliptic functions imply that
$\pInc(\theta_0)=(\varphi_0,r_0)$ and $\pInc(\theta_1)=(\varphi_1,r_0)$.
We know that $\widetilde B_e^f$ is of the form \Eq{LiftedBilliard}
with $0<\rhoin(e,f)<1$. Moreover, \Eq{qLift} implies that $\pInc \circ \widetilde B_e^f=B|_{\Lambda(e)} \circ \pInc$,
and thus
\[
	\pInc\left(\widetilde B_e^f(\theta_0)\right)=B|_{\Lambda(e)} \left(\pInc(\theta_0)\right)
		=B|_{\Lambda(e)} \left(\varphi_0,r_0\right)=(\varphi_1,r_0)=\pInc(\theta_1).
\]
Therefore, there exists an integer $m\in\mathbb{Z}$ such that
\(
	\theta_0+\rhoin(e,f)=\widetilde B_e^f(\theta_0)=\theta_1+m.
\)

Since $0<\omega<\pi/2$, we have that $0<F(\omega,e)<K(e)$. Then \Eq{theta12Def} implies
$0<\theta_1-\theta_0<\tfrac{1}{2}$. We conclude that
\[
	-\tfrac12 < m = \rhoin(e,f) - (\theta_1 -\theta_0) = \rhoin(e,f) - \frac{F(\omega,e)}{2K(e)} < 1 .
\]
Thus the integer $m=0$ and this directly gives \Eq{BilliardRotationNumber}.
\end{proof}

\begin{remark}
The geometric idea of the proof above is to choose two points $z_0,z_1\in E(f)$ that are related by symmetry and
that have the same value of the canonical momentum $r_0$. The implication is that the
line $\overrightarrow{ z_0 z_1}$ is horizontal. Choosing
$\varphi_0,\varphi_1$ as in \Eq{twoangles}, we define $z_0=\gamma(\varphi_0)$ and $z_1=\gamma(\varphi_1)$.

This implies that
\[
	z_0=\tfrac{1}{f}\big(\sin\omega,\sqrt{1-f^2}\,\cos\omega\big),\qquad
	z_1=\tfrac{1}{f}\big(-\sin\omega,\sqrt{1-f^2}\,\cos\omega\big).
\]
A simple computation from \Eq{BilliardRotationNumber} shows that $(\sqrt{1-f^2}/f)\,\cos\omega=\sqrt{1-e^2}/e$. We conclude that the segment
$\overrightarrow{ z_0 z_1}$ is tangent to $E(e)$ at the point $(0,\tfrac{1}{e}\sqrt{1-e^2})$, implying that $E(e)$ is a caustic for the orbit.
\end{remark}

\begin{remark}\label{rem:limits}
Some of the level sets of the rotation number \Eq{BilliardRotationNumber} were shown in \Fig{ConstantRho}.
Note that the curves in \Fig{ConstantRho} begin at $(0,0)$ and grow monotonically, ending at $(1,1)$.
The bounding values follow from two simple limits corresponding to the degenerate cases:
\[
 	\lim_{f\to e^-}\rhoin(e,f)= \lim_{e\to f^+}\rhoin(e,f)=0,\qquad
	\lim_{f\to 0^+}\rhoin(e,f)=\lim_{e\to 1^-}\rhoin(e,f)=\tfrac12.
\]
The last limit is not trivial---it is shown in \App{SingularLimit}.
In \Sec{ExplicitPorisms}, and in particular in \Lem{porism}, we will give an explicit formula that
will explain the monotonicity.  In \App{Derivatives}, we prove that \Eq{BilliardRotationNumber} satisfies
$\displaystyle\frac{\partial \rhoin}{\partial e}>0$ and $\displaystyle\frac{\partial \rhoin}{\partial f}<0$, see  \Lem{rotgrowth}.
\end{remark}

Using the formula \Eq{BilliardRotationNumber} and geometrical considerations, one can find polynomial relations between
the eccentricities $e$ and $f$ that correspond to low-order rational rotation numbers.
It is generally difficult, however, to construct such polynomial relations for these level sets.
In the next section we develop a general method to establish the relation between $e$ and $f$ to obtain a
given fixed rotation number.

\section{Poristic Confocal Ellipses and Gauss symmetry}\label{sec:Poristic}

The problem of identifying pairs of ellipses for which the rotation number of a Poncelet
map is rational is known as the problem of Poncelet porisms.
In \Sec{PoristicParams} we recall this definition and---for the
the map $B_e^f$---recall some of the explicit polynomials in $(e,f)$ that correspond to low-order rational
values for $\rhoin(e,f)$.
Then in \Sec{ExplicitPorisms} we obtain a general formula for the eccentricity $f$
of the outer ellipse so that the Poncelet map for a given inner ellipse has a specified rotation number.
In \Sec{HalfRho} we use this formula to find a transformation of $f$ that halves the rotation number.
Finally in \Sec{GaussSym} we show that the billiard rotation number has a symmetry that we call \emph{Gauss symmetry}.

\subsection{Poristic Parameters}\label{sec:PoristicParams}
We denote the subset of \Eq{DeltaDefine} with a given rotation number $\ell$ by
\beq{DeltaEll}
	\poris{\ell} = \{ (e,f) \in \Delta : \rhoin(e,f) = \ell\}.
\eeq
Several such curves were shown in \Fig{ConstantRho}.
The non-empty level sets of $\rhoin^{-1}(\bQ)$ correspond to porisms:

\begin{defn}[Poristic Ellipses]
A pair of confocal ellipses $E(e)$ and $E(f)$ with eccentricities $0 < f < e < 1$
is \textbf{poristic} if the rotation number \Eq{BilliardRotationNumber} of the Poncelet map is rational.
For each $q\in\bQ$, $\poris{q}$ will be called a \textbf{poristic parameter set}.
\end{defn}

Thus it is clear that a poristic pair $(e,f)$ will generate a
closed polygon with vertices on the outer ellipse, i.e., that $(B_e^f)^m$ is the identity for some $m\in\bN$.
The first few poristic sets are given by the polynomials
\begin{align}
	\poris{1/3}&=\{(e,f)\in\Delta:e^2+2 e f^3-2 e f-f^4=0\},\label{eq:OneThird}\\
	\poris{1/4}&=\{(e,f)\in\Delta:e^2+f^4-2 f^2=0\}, \label{eq:OneFourth}\\
	\poris{1/6}&=\{(e,f)\in\Delta:4 e^4 f^2-3 e^4-6 e^2 f^4+4 e^2 f^2+f^8=0\}. \label{eq:OneSixth}
\end{align}
We verify \Eq{OneFourth} below, and  in \App{Porisms}
we verify the set \Eq{OneThird} and find $\Delta_{1/5}$ and $\Delta_{2/5}$.
The set \Eq{OneSixth} is obtained in \Sec{HalfRho}.

\begin{example}[Period Four Porism] \label{ex:OneFourth}
To show that $\rhoin(e,f) = \tfrac14$ for \Eq{OneFourth}, we substitute the polynomial into \Eq{BilliardRotationNumber}
to obtain
\[
 	\omegain(e,f)|_{\Delta_{1/4}} = \arcsin\left(\frac{1}{\sqrt{2-f^2}}\right)
	 		= \arcsin\left(\frac{1}{\sqrt{1+e'}}\right) ,
\]
with $e' := \sqrt{1-e^2}$. Then using \cite[Form. 110.04]{Byrd1971}
\beq{raro}
	F(\omegain(e,f),e)=\tfrac12K(e),
\eeq
we see that \Eq{BilliardRotationNumber} implies that the rotation number is $\rhoin(e,f)=\tfrac14$.

It is interesting to contrast this with a direct computation of the orbit.
If we let $z_0 = (x_0,y_0) = (1/e, -e'/e)$, $z_1=(-x_0, y_0)$, $z_2=(-x_0, -y_0)$, $z_3=(x_0, -y_0)$,
and $z_4=z_0$, then  \Eq{OneFourth}
implies that $z_k \in E(f)$,
where $E(f)$ is the outer ellipse given by \Eq{Ellipse}.
A direct computation shows that
 each directed segment $\overrightarrow{z_k\,z_{k+1}}$ is tangent
to the ellipse $E(e)$, and has positive orientation.
This implies that the circle map
satisfies $B_e^f(z_k) = z_{k+1}$ so that and $(B_e^f)^4=id$ and $\rhoin(e,f)=\tfrac14$.

\end{example}

\subsection{Explicit Porisms}\label{sec:ExplicitPorisms}

More generally, we can obtain an explicit formula for the
eccentricity $f$ of the outer ellipse for a given caustic $e$ and any given rotation number $\rhoin(e,f)=\ell$.

\begin{lem}[Explicit porism]\label{lem:porism}
If $(e,f)\in \Delta$, \Eq{DeltaDefine}, and $\ell\in(0,1/2)$,
then $\rhoin(e,f)=\ell$ if and only if
 \beq{poristic}
	f=e\, \sn(K(e)(2\ell +1),e)
	=e\,\cd(2K(e)\ell,e).
 \eeq
Thus the map $e\mapsto (e, e\,\cd(2K(e)\ell,e))$, for $0<e<1$, parameterizes the
curve $\Delta_\ell$.
\end{lem}

\begin{proof}
We will show only one direction. The converse is similar.
To compute $\omegain(e,f)$ using \Eq{BilliardRotationNumber}, note that \Eq{poristic} gives
\begin{align*}
	 \frac{\displaystyle e^2-f^2}{e^2(1-f^2)}
	 	&=\frac{1-\cd^2(u,e)}{1-e^2\cd^2(u,e)}
		=\sn^2(u,e),
\end{align*}
where $u = 2K(e)\ell$.
Given that $\ell\in(0,1/2)$, then $0<\am(2K(e)\ell,e)<\frac{\pi}{2}$, and \Eq{BilliardRotationNumber} implies that
\(
	\omegain(e,f)
	=\arcsin(\sn(2K(e)\ell,e))
	= \am(2K(e)\ell,e).
\)
This implies $F(\omegain(e,f),e)=2K(e)\ell$, and the rotation number \Eq{BilliardRotationNumber} becomes
$\rhoin(e,f)=\ell$.
\end{proof}

An application of the Gauss transformation of elliptic integrals to the relation \Eq{poristic} is given in \Sec{GaussSym}.

\begin{remark}\label{rem:DeltaFoliation}
In \Fig{ConstantRho}, we observed that the set $\Delta$ is foliated by the
curves $\Delta_\ell$, for $\ell\in(0,1/2)$. In the figure, each
$\Delta_\ell$ is the graph of a function of $e$ as in \Lem{porism}.
Each of these curves appears to start at $(0,0)$ and end at $(1,1)$. This can be explained
from \Eq{poristic} using the limits
 \[
 	\lim_{e\to 0^+}e\,\cd(2K(e)\ell,e)=0,\qquad
 	\lim_{e\to 1^-}e\,\cd(2K(e)\ell,e)=1 ,
 \]
that are valid for $\ell\in(0,1/2)$. Since that $K(e)\to\infty$ when $e\to 1^- $,
the second limit is non-trivial and requires appropriate asymptotic
analysis, as we see in \App{SingularLimit}.
In addition, we can use \Lem{paramgrowth} to show that
 $\displaystyle\frac{\partial f}{\partial \ell}<0$ and $\displaystyle\frac{\partial f}{\partial e}>0$.
More details can be found in \App{Derivatives}.
\end{remark}

\subsection{Half-Rotation Number}\label{sec:HalfRho}

The explicit poristic formula in Lemma \ref{lem:porism} can be used with double-angle formulas for
Jacobi elliptic functions to find a transformation of
$(e,f) \in \Delta$ that will halve the rotation number.

\begin{lem}[Half-rotation]\label{lem:HalfRotation}
 Given $\ell\in(0,1/2)$,
 $0<e<1$ and $f_0=e\,\cd(2K(e)\ell,e)$, define
 \beq{HalfRotation}
f_1=\sqrt{\frac{e \left(e+f_0\right)}{1+e f_0+s}},\, \quad
	f_2=\sqrt{\frac{e \left(e-f_0\right)}{1-e f_0+s}},\, \quad s \equiv \sqrt{(1-e^2)(1-f_0^2)}.
 \eeq
Then
$f_1=e\,\cd(K(e)\ell,e)$, $f_2=e\,\cd(K(e)(1-\ell),e)$, and
\[
	\rhoin(e,f_1)=\tfrac12\rhoin(e,f_0), \qquad \rhoin(e,f_2)=\tfrac12-\tfrac12\rhoin(e,f_0),
\]
so that when $(e,f_0)\in\Delta_\ell$, then $(e,f_1)\in\poris{\ell/2}$, and $(e,f_2)\in\poris{1/2-\ell/2}$.
\end{lem}
\begin{proof} Using \Eq{poristic}, the double angle formulas for Jacobi elliptic functions
\cite[Form.124.01]{Byrd1971} give
 \beq{half1}
 	\frac{f_0}{e}= \cd(2K(e)\ell,e)=\frac{u-v}{1-e^2\,uv},
 \eeq
 where $u=\cd^2(K(e)\ell,e)$ and
 \beq{half2}
 	v= \sn^2(K(e)\ell,e) = \frac{1-\cd^2(K(e)\ell,e)}{1-e^2\,\cd^2(K(e)\ell,e)}=\frac{1-u}{1-e^2u}.
 \eeq
 Using \Eq{half1} and \Eq{half2} we find that
\beq{uveq}
	\frac{f_0}{e}=\frac{-e^2 u^2+2 u-1}{e^2 u^2 -2 e^2 u+1}=\frac{e^2 v^2-2 v+1}{e^2 v^2 -2 e^2 v+1} ,
\eeq
which can be solved for $u$ to give
 \[
 	u =\frac{1+e f_0- s}{e \left(e+f_0\right)}
	 =\frac{e+f_0}{e \left(1+e f_0+ s\right)}.
\]
Since $\ell\in(0,1/2)$, then $\cd(K(e)\ell,e)>0$, so that by \Eq{HalfRotation},
$
	f_1 = e\sqrt{u} =  e\,\cd(K(e)\ell,e).
$

Similarly, solving \Eq{uveq} for $v$, gives
 \[
 	v =\frac{1-e f_0- s}{e \left(e-f_0\right)}
	 =\frac{e-f_0}{e \left(1-e f_0+ s\right)}.
\]
Since $\ell\in(0,1/2)$, $\sn(K(e)\ell,e) > 0$, 
so that by \Eq{HalfRotation},
$
	f_2 = e\sqrt{v}= e\, \sn(K(e)\ell,e) = e\,\cd(K(e)(1-\ell),e).
$
These results clearly imply that $0<f_1,f_2<e$ and hence $(e,f_1),(e,f_2)\in\Delta$.
Finally, Lemma \ref{lem:porism} implies that $\rhoin(e,f_1)=\tfrac12\rhoin(e,f_0)$,
and $\rhoin(e,f_2) = \tfrac12- \tfrac12\rhoin(e,f_0)$.
\end{proof}

\begin{example}
An example of the application of \Lem{HalfRotation} is shown in \Fig{halfrotnew}.
Panel (a) shows a pair $E(e)$ and $E(f_0)$ with $(e,f_0)\in \poris{2/7}$.
Using \Eq{HalfRotation}, gives a new outer ellipse in panel (b) with
$\rhoin(e,f_2) = \tfrac12-\tfrac17 = \tfrac{5}{14}$.
\end{example}

\InsertFigTwo{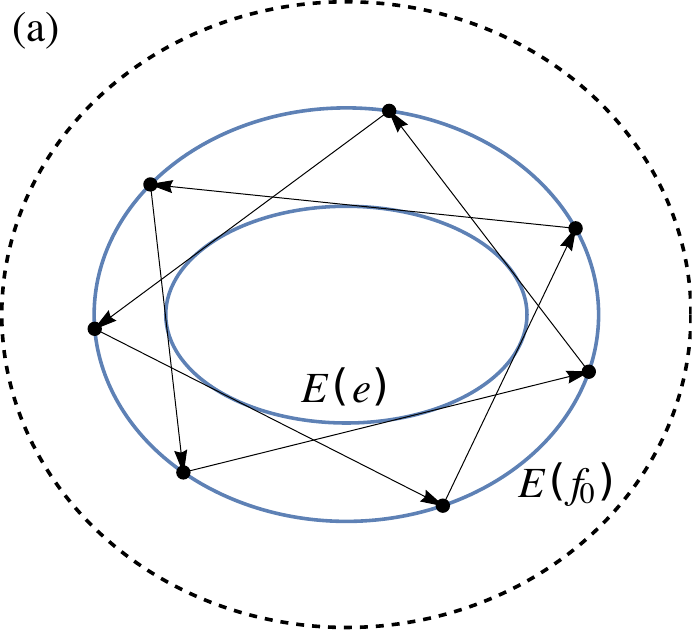}{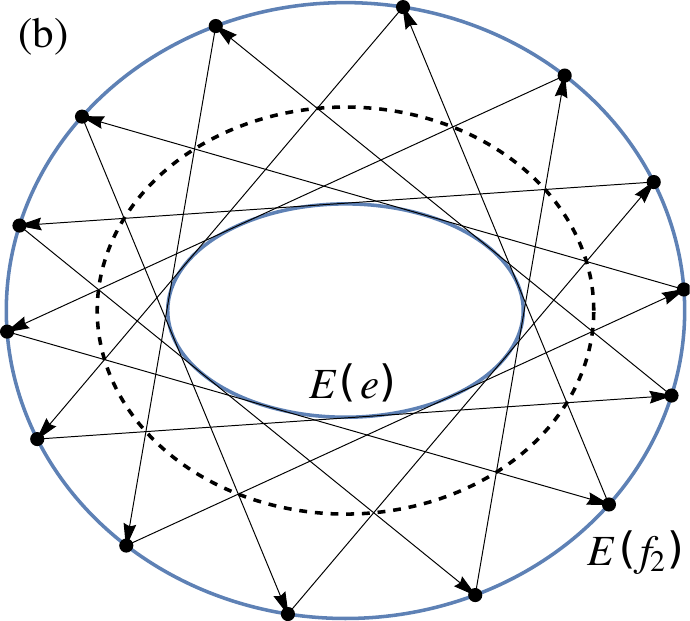}{An example of the half-rotation formulas \Eq{HalfRotation}.
(a) A pair of poristic ellipses $E(e),E(f_0)$ with $(e,f_0) = (0.8, 0.572851)$, and $\rhoin(e,f_0)=2/7$.
(b) Mapping $(e,f_0)$ to $(e,f_2)=(0.8, 0.419316)$, gives
a pair $E(e),E(f_2)$ that is still poristic with
$\rhoin(e,f_2)=5/14$. The dashed ellipses $E(f_2)$ in (a) and $E(f_0)$ in (b) are drawn for comparison.
}{halfrotnew}{2.2in}

\begin{example}
The half-rotation formula \Eq{HalfRotation}, in conjunction with \Eq{OneThird} immediately gives \Eq{OneSixth}.
\end{example}

\subsection{Gauss Symmetry}\label{sec:GaussSym}

In this section we obtain a symmetry for the rotation number \Eq{BilliardRotationNumber} and show that
any point $(e,f)$ with a given rotation number can be transformed into an infinite number of
such points with the same rotation number.

Landen and Gauss discovered some useful transformations of elliptic integrals
which can be used to obtain a symmetry of the billiard rotation number.
We will call this map \emph{Gauss symmetry.}

\begin{lem}[Gauss symmetry]\label{lem:GaussSym}
 Define the function $\fg: \bR^+ \to (0,1]$ by
\[
	 \fg(e)=\frac{2\sqrt{e}}{1+e}.
\]
If $G:\Delta \to\Delta$ is the transformation
\beq{trans}
	G(e,f)=\left(\fg(e), \fg\left(f^{2}\cdot e^{-1}\right)\right)
		 =\left(\frac{2 \sqrt{e}}{1+e},\frac{2 \sqrt{e} f}{f^2+e}\right),
\eeq
then the rotation number $\rhoin$ given by \Eq{BilliardRotationNumber} is invariant under $G$; i.e., $\rhoin\circ G=\rhoin$.
\end{lem}

\begin{proof}
Fix eccentricities $0<f<e<1$. Using \Eq{BilliardRotationNumber}, we define $\omega_1=\omegain(e,f)$ and
$\omega_2=\omegain(G(e,f))$. Clearly $0<\omega_1,\omega_2<\tfrac{\pi}{2}$,
and a direct computation shows that
\[
	\sin(\omega_2)(1+e\sin^2\omega_1)=(1+e)\sin\omega_1.
\]
The Gauss transformation 
\cite[Formula 8.125.3]{Gradshteyn65}
of the elliptic integral \Eq{EllipticIntegral} then implies that
 \[
 	F(\omega_2, \fg(e)) = \int_0^{\omega_2}\frac{d\,\tau}{\sqrt{1-\fg(e)^2\sin^2\tau}}
	 = (1+e) F(\omega_1,e).
 \]
 In particular, we note that $K(\fg(e)) = F(\tfrac{\pi}{2},\fg(e))=(1+e)F(\tfrac{\pi}{2},e) = (1+e)K(e)$.

 Using \Eq{BilliardRotationNumber}, we conclude that
 \[
 	\rhoin(G(e,f))=\frac12\,\frac{F(\omega_2,\fg(e))}{K(\fg(e))}
		=\frac12\,\frac{(1+e)\,F(\omega_1,e)}{(1+e)\,K(e)}= \rhoin(e,f),
\]
which is the promised result.
\end{proof}

As a consequence of this Gauss symmetry, given $(e,f)\in\Delta$, the orbit of
$G$ gives an infinite number of points in $\Delta$ with the same rotation number.
An example of the application of \Eq{trans} is shown in \Fig{gausssym}.
Panel (a) shows a pair of ellipses and an orbit for $\rhoin(e_0,f_0) = \tfrac15$. Under $G$,
this transforms to eccentricities $(e_1,f_1)$, as shown in panel (b),
with the same rotation number.

\InsertFigTwo{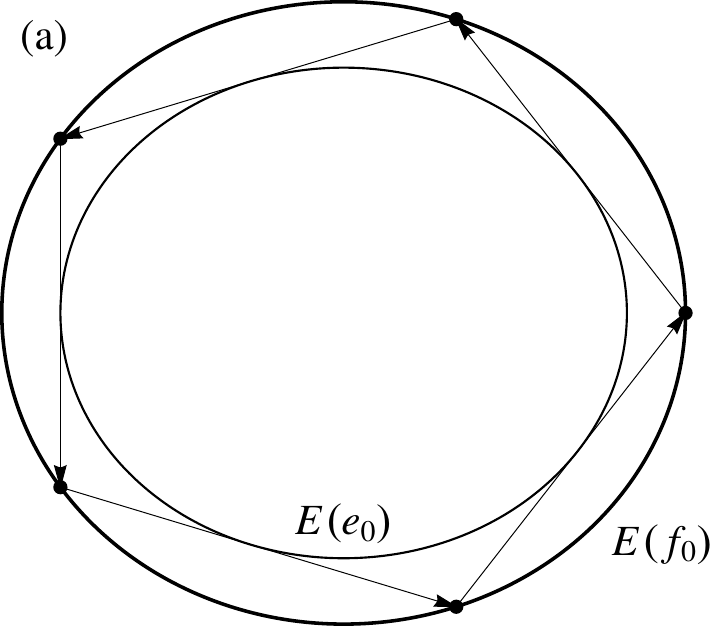}{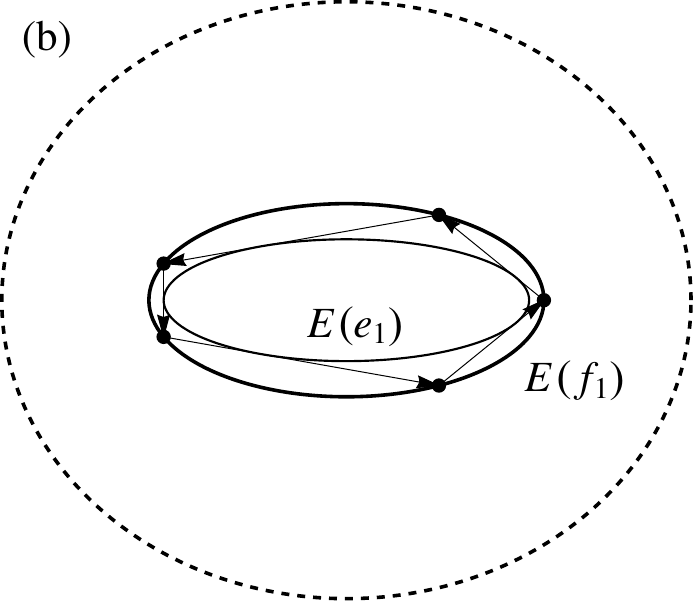}{An example of Gauss symmetry with $\rho = \tfrac15$.
(a) A pair of ellipses $E(e_0),E(f_0)$ with $(e_0,f_0) = (0.6, 0.503246)$,
and an orbit of the resulting Poncelet map.
(b) Gauss symmetry transforms $(e_0,f_0)$ to $(e_1,f_1)= G(e_0,f_0) = (0.968246, 0.913706)$, giving
a pair $E(e_1),E(f_1)$ that still has
$\rhoin(e_1,f_1)=\tfrac15$. The dashed ellipse is $E(f_0)$, drawn for comparison.
}{gausssym}{2.2in}
%

Note that an immediate consequence of \Lem{GaussSym} is the following.

\begin{cor} The poristic set $\poris{\ell}$ \Eq{DeltaEll} is invariant under $G$: if $(e,f) \in \poris{\ell}$,
then so is $G(e,f)$.
\end{cor}

In particular, if a polynomial expression---like that in \Eq{OneThird}---defines
a poristic parameter set, the polynomial must be invariant under the symmetry.


\section{Pencils of Ellipses}\label{sec:EllipticPencils}

To prepare for \Sec{PonceletMaps}, where we will generalize the results of \Sec{BilliardMaps}
for confocal ellipses to a Poncelet map between any pair of nested ellipses,
we develop in this section some notation and coordinate transformations.

In \Sec{Standing} we recall the definition of a \emph{pencil} as the family of
linear combinations of the implicit equations for a pair of ellipses.
For our analysis, we will fix the outer ellipse $\cC_o$, and choose an element $\cC_\bnu$
of the pencil that is in its interior. We will show in \Sec{Diagonalization} that there is
a linear transformation of any pencil to a diagonal form that is characterized by three eigenvalues.

For each $\bnu$, the corresponding Poncelet map $P_\bnu$ \Eq{PencilPoncelet} is the map that uses
$\cC_i=\cC_\bnu$ as the interior ellipse.
In \Sec{PencilE} we obtain a key invariant that we call the \emph{eccentricity of the pencil}.
The dynamics of any of the Poncelet maps $P_\bnu$ can be characterized in terms of Jacobi elliptic functions that
have moduli give by this eccentricity.
In \Sec{Projective}, we will show that any pencil can be transformed,
by a projective transformation, to a standard pencil for which the outer ellipse becomes the unit circle and the inner ellipse is centered at the origin with its
principle axes aligned with the coordinate axes.
This transformation can also be applied to the Poncelet map, giving a conjugate map.

\subsection{Standing Assumptions}\label{sec:Standing}

A conic $\cC$ in the plane can be represented by a real, $3\times 3$ symmetric matrix $C$---using
homogeneous coordinates---as the set \Eq{QuadraticForm}. In order that $\cC$ be an ellipse we assume
that the upper-left $2 \times 2$ minor of $C$ is positive definite.\footnote
{Of course, we could replace $C$ by $-C$ to obtain these signs if necessary.}
 i.e., $\bw^T C \bw > 0$, whenever  $\bw = (x,y,0)^T \neq \mathbf{0}$.
That is, we will assume that the quadratic terms of \Eq{QuadraticForm} are positive definite.
As a consequence, to have a nonempty $\cC$, the
full $3\times 3$ matrix $C$ must be indefinite, thus $\det{C} < 0$.
We will refer to matrices with these properties as having the \emph{ordered} signature $(+,+,-)$.
In this case it is clear that the interior of the ellipse \Eq{QuadraticForm} is the
set of points $(x,y)\in\bR^2$ such that $\bu^T C \bu < 0$, where $\bu=(x,y,1)$.

If $C_1$ and $C_2$ are two symmetric matrices, with corresponding
conics $\cC_1 $ and $\cC_2$, respectively, the associated pencil $\cC_\bnu=\cC_{(\nu_1,\nu_2)}$ is the family \Eq{PencilSet}
 for $\bnu =(\nu_1,\nu_2)\in \bR^2$.
The interior of $\cC_{\bnu}$ is
\[
	\interior(\cC_{\bnu}) = \left\{ (x,y) \in \bR^2 : \bu^T ( \nu_1C_1 + \nu_2 C_2) \bu < 0 \,,\, \bu = (x,y,1)^T \right\} .
\]
An example is sketched in \Fig{pencil}.
Note that $\cC_\bnu =\cC_{\bmu}$ if the vectors $\bnu$ and $\bmu$ are non-zero and parallel, i.e.,
$
	\cC_{\delta \bnu} = \left\{ \delta \bu^T ( \nu_1C_1 + \nu_2 C_2) \bu = 0 \right\} = \cC_\bnu,
$
for any $\delta \neq 0$. Thus the pencil is defined on the projective line $\bnu \in \bP^1$.
 We will restrict to parameters for which $\cC_\bnu\subset \interior(\cC_1)$.

\InsertFig{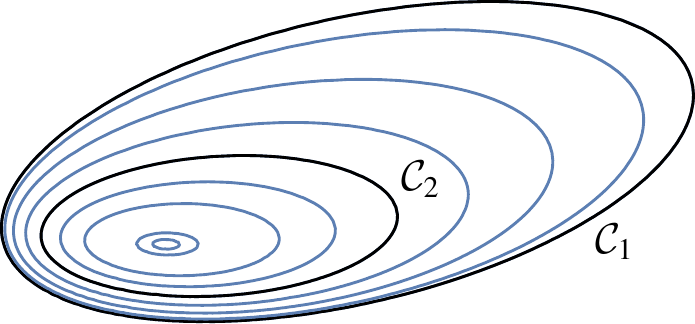}
{A pencil of ellipses generated by an outer ellipse $\cC_1$ and an inner ellipse $\cC_2$.
 Each pair of ellipses of the pencil will generate a Poncelet map. }
{pencil}{3 in}

For the analysis below, we make the following assumptions on
the matrices in the pencil \Eq{PencilSet}:
 \begin{enumerate}[label={(SA\arabic*)}]
 \item\label{ass:SA1}The symmetric $3 \times 3$ matrices $C_1$ and $C_2$ have ordered signature $(+,+,-)$:
 their upper-left $2\times 2$ minors are positive definite, and $\det{C_i} < 0$. In particular,
 \beq{SA1}
 	(x,y,0)C_i(x,y,0)^T>0, \quad \mbox{for all } (x,y)\neq(0,0) \;.
\eeq
 \item\label{ass:SA2}
 	The matrix $C_1^{-1}C_2$ 
 has three positive,
	distinct eigenvalues.
\item\label{ass:SA3}
	The ellipse $\cC_2$ is inside $\cC_1$; that is, whenever $\bu \neq 0$
	and $\bu^TC_2\bu =0$, then $\bu^TC_1\bu < 0$.
 \end{enumerate}

\begin{remark}
 The concept of \emph{generalized eigenvalues of a pencil}, discussed in \cite{Parlett1998,Stewart90} is related
 to our construction.
 Given a pair of matrices $C_1,C_2$, the generalized eigenvalues of the pencil generated by
 $(C_1,C_2)$, are the roots of the polynomial $\det(\lambda \,C_1-C_2)$. Since $\ref{ass:SA1}$ implies that
 $C_1$ is invertible, the eigenvalues mentioned in $\ref{ass:SA2}$ are the generalized eigenvalues of the pencil.
\end{remark}

\begin{remark}\label{rem:trickSA3}
 Standing assumption \ref{ass:SA3} follows because by \Eq{SA1}, whenever $\bu=(x,y,z)$ and $\bu^TC_2 \bu = 0$, we
must have $z \neq 0$. Thus on the ellipse \Eq{QuadraticForm}, we can scale the vector $(x,y,1)^T$
by $z$ to obtain $\bu^TC_2 \bu = 0$ for the more general vector $\bu = (zx,zy,z)^T$.
Similarly when $(x,y,1) C_1 (x,y,1)^T < 0$, the corresponding scaled
equation $\bu^TC_1 \bu < 0$ as well since each term in the quadratic form  is scaled by $z^2 > 0$.
\end{remark}

\begin{remark}\label{rem:MatrixScaling}
There is some freedom in the choice of matrices to represent the ellipses.
For example, if the pair $(C_1,C_2)$ satisfies the assumptions, then
whenever $\delta, \gamma>0$, the pair $(\delta C_1,\gamma C_2)$ does as well.
Indeed, the eigenvalues of \Ass{SA2} are homogeneous of degree one in $C_1$, and
homogeneous of degree minus one in $C_2$.
\end{remark}

\subsection{Simultaneous Diagonalization}\label{sec:Diagonalization}
Here we will show that there is a transformation of the pair of matrices
that satisfy the assumptions of \Sec{Standing} to diagonal form.

\begin{lem}\label{lem:reduction}
If the matrices $C_1, C_2$ obey \Ass{SA1}-\ref{ass:SA3} then
there exists a nonsingular matrix $M$ such that
 \beq{SylvesterM}
	M^TC_1M= \diag(1,1,-1),
		\quad \mbox{and } \quad
	M^TC_2M= \diag(\lambda_1,\lambda_2,-\lambda_3)\;,
 \eeq
where $\{\lambda_1,\lambda_2,\lambda_3\}$ are the eigenvalues of $C_1^{-1}C_2$, and
\beq{EVOrdering}
	\lambda_1 > \lambda_2 > \lambda_3 >0 \;.
\eeq
\end{lem}

\begin{proof}
To obtain \Eq{SylvesterM}, \Ass{SA2} implies there exists
an invertible matrix $Q_0$ that diagonalizes $C_1^{-1}C_2$; that is
\[
	Q_0^{-1}C_1^{-1}C_2\,Q_0={D} \;, 
\]
where ${D}$ is diagonal and positive definite.
This clearly implies 
that $Q_0^TC_2Q_0= (Q_0^TC_1Q_0){D}$. Noting that $Q_0^TC_2Q_0$ and $Q_0^TC_1Q_0$
are symmetric, then the commutator $[Q_0^TC_1Q_0, {D}] = 0$. Since the entries of ${D}$ are assumed distinct,
then $Q_0^TC_1Q_0$ must be diagonal. Finally $Q_0^TC_2Q_0$ is the product of diagonal matrices, so it too is diagonal.

Recall that Sylvester's law of inertia implies that the signature does not change under a congruency.
Thus by \Ass{SA1} there are exactly two positive and one negative entries on the diagonals of
the congruent matrices $Q_0^TC_1Q_0$ and $Q_0^TC_2Q_0$, and the negative entry is in the same position
in both, since ${D}$ is positive definite.
With an appropriate elementary permutation matrix $Q_1$, we can swap a pair of rows
and the same pair of columns to order these entries, i.e.,
\[
	Q_1^T(Q_0^T C_1Q_0)Q_1= \diag(\alpha_1, \alpha_2, \alpha_3),
\]
where $\alpha_1,\alpha_2>0$, and $\alpha_3<0$.
If we let
\(
	T= \diag\left(\tfrac{1}{\sqrt{\alpha_1}},\tfrac{1}{\sqrt{\alpha_2}}, \tfrac{1}{\sqrt{-\alpha_3}} \right)
\)
and $M=Q_0Q_1T$, then
\[
	M^TC_1M = T^TQ_1^TQ_0^T C_1 Q_0Q_1T = T\, \diag(\alpha_1, \alpha_2, \alpha_3)\, T = \diag(1,1,-1),
\]
giving the form \Eq{SylvesterM}.
Finally, since $C_2 = C_1 Q_0{D} Q_0^{-1}$, a simple computation then gives
\begin{align*}
	M^T C_2 M &= M^T C_1 Q_0{D} Q_0^{-1} Q_0Q_1T = M^T C_1 M M^{-1}Q_0{D} Q_1T \\
	 		 &= \diag(1,1,-1) T^{-1} (Q_1^{T} {D} Q_1 ) T,
\end{align*}
where, since $Q_1$ is an elementary permutation, $Q_1^{-1} = Q_1^T$.
Now defining the reordered eigenvalues so that
$T^{-1}Q_1^{T} {D} Q_1 T= \diag(\lambda_1,\lambda_2,\lambda_3)$, gives
\[
	 M^T C_2 M = \diag(1,1,-1) \diag(\lambda_1,\lambda_2,\lambda_3),
\]
which is equivalent to \Eq{SylvesterM}. Clearly, $\{\lambda_1,\lambda_2,\lambda_3\}$
are the eigenvalues of $C_1^{-1}C_2$.

It remains to show that the ordering \Eq{EVOrdering} holds. Suppose that $(w_1,w_2)\neq(0,0)$ are arbitrary and
\[
	\bw =(w_1,w_2,w_3)= \left(w_1,w_2,\sqrt{\frac{\lambda_1w_1^2+\lambda_2w_2^2}{\lambda_3}} \right) .
\]
If $\bu = M \bw$, then both $\bw$ and $\bu$ are non-zero and \Eq{QuadraticForm} implies that
\[
	\bu^T C_2 \bu = \bw^T M^{T} C_2 M \bw = \lambda_1 w_1^2 + \lambda_2 w_2^2 - \lambda_3 w_3^2=0,
\]
so that $\bu$ is a point on $\cC_2$.
By \Ass{SA3}, $\cC_2 $ is in the interior of $\cC_1$; therefore we must have
$\bu^T C_1 \bu < 0$. Using the definitions of $\bu$ and $\bw$, this inequality transforms to
\[
	\bu^T C_1 \bu^T = \bw^T M^{T} C_1 M \bw = w_1^2 + w_2^2 - w_3^2 = \left(1-\frac{\lambda_1}{\lambda_3}\right)w_1^2
					+ \left(1-\frac{\lambda_2}{\lambda_3}\right)w_2^2 < 0.
\]
Since this must be true for any $(w_1,w_2)\neq(0,0)$, we must have $\lambda_3 < \lambda_2, \lambda_1$.
Without loss of generality we can assume that $\lambda_2 < \lambda_1$, since by \Ass{SA2} the eigenvalues
are distinct; indeed this ordering can be obtained by changing the permutation matrix $Q_1$ if needed.
This concludes the proof.
\end{proof}

\begin{remark}
 The previous lemma can be regarded as a simultaneous diagonalization result in the
 context of generalized eigenvalues. Similar results about the diagonalization---and
 the order of the eigenvalues---are known \cite{Parlett1998,Stewart90} for the case
 in which one of the matrices is positive definite.
 In our case the positive definite condition is replaced by \Ass{SA1}-\ref{ass:SA3}.
\end{remark}

Given the normal forms \Eq{SylvesterM}, we next obtain conditions on the
parameters $(\nu_1,\nu_2)$, so that members of the family
\Eq{PencilSet} also satisfy the standing assumptions.

\begin{lem}\label{lem:restrict}
Suppose that $C_1$ and $C_2$ satisfies \Ass{SA1}-\ref{ass:SA3}, and that $\bnu=(\nu_1,\nu_2)$ satisfy
\Eq{nuRestrictions}, e.g.,
$
 	\nu_2>0 \mbox{ and } \nu_1+\lambda_3\nu_2>0
$.
Then the pair of matrices $(C_1,C_\bnu)$, with $C_\bnu=\nu_1C_1+\nu_2C_2$, also satisfies the standing assumptions.
In this case, we say that $\bnu=(\nu_1,\nu_2)$ is a \emph{valid parameter} vector.
\end{lem}

\begin{proof}
Using \Lem{reduction}, there is a matrix $M$ giving \Eq{SylvesterM}
where $\lambda_1>\lambda_2>\lambda_3>0$.
This implies that
\[
	C_2-\lambda_3C_1= M^{-T} \diag(\lambda_1-\lambda_3,\lambda_2-\lambda_3, 0) M^{-1}.
\]
Since both $\lambda_1-\lambda_3$ and $\lambda_2-\lambda_3>0$, the matrix
$C_2-\lambda_3C_1$ is positive semi-definite.
Thus $\bu^T(C_2-\lambda_3C_1)\bu\geq0$, for all $\bu\in\bR^3$.
We can write the matrix $C_\bnu$ of \Eq{PencilSet} as
\[
	C_\bnu=\nu_1C_1+\nu_2C_2=(\nu_1+\lambda_3\nu_2)C_1+\nu_2(C_2-\lambda_3C_1).
\]
Note that when $\bnu$ satisfies the restrictions \Eq{nuRestrictions}
then since $C_2-\lambda_3C_1$ is positive semi-definite, and $C_1$ satisfies \Ass{SA1},
then $C_\bnu$ must have a positive definite upper-left minor. That is,
$(x,y,0)C_\bnu (x,y,0)^T >0$ whenever $(x,y)\neq(0,0)$.
This proves that $C_\bnu$ satisfies \Ass{SA1}.

To verify \Ass{SA2}, it is enough to notice that the eigenvalues of $C_1^{-1}C_\bnu$ are
$\{\widetilde{\lambda}_1,\widetilde{\lambda}_2,\widetilde{\lambda}_3\}$, where
$\widetilde{\lambda}_i=\nu_1+\lambda_i\nu_2$. Then \Eq{EVOrdering} and \Eq{nuRestrictions} imply
that $0<\widetilde{\lambda}_3<\widetilde{\lambda}_2<\widetilde{\lambda}_1$. This proves \Ass{SA2}.

Finally, to prove \Ass{SA3}, assume that $\bu\neq\mathbf{0}$ satisfies
$\bu^TC_\bnu\bu=0$. We want to show that $\bu^T C_1 \bu < 0$.
Since $C_2-\lambda_3C_1$ is positive semi-definite,
 \[
 	(\nu_1+\lambda_3\nu_2)\bu^TC_1\bu
	= \bu^T (-\nu_2 C_2 + \lambda_3\nu_2 C_1) \bu
	= -\nu_2\bu^T(C_2-\lambda_3C_1)\bu\leq0.
 \]
If it were the case that $\bu^TC_1\bu=0$, then the above inequality implies that $\bu^TC_2\bu=0$.
However, then \Ass{SA3} would imply that
$\bu^TC_1\bu<0$. This is a contradiction, so the only possibility is
that $\bu^TC_1\bu<0$, and this implies \Ass{SA3}.
 \end{proof}

The region defined in \Eq{nuRestrictions} is sketched in \Fig{PencilRegion}.
Note that it is the cone generated by the two extreme cases $\cC_o = \cC_1$ and $\cC_\infty$ that correspond to
the vectors and $(1,0)$ and $(-\lambda_3,1)$, respectively (see \Sec{Projective} below).
Moreover, by \Lem{restrict}, $\cC_\bnu\subset \interior(\cC_1)$, whenever $\bnu=(\nu_1,\nu_2)$ is in this region.

\InsertFig{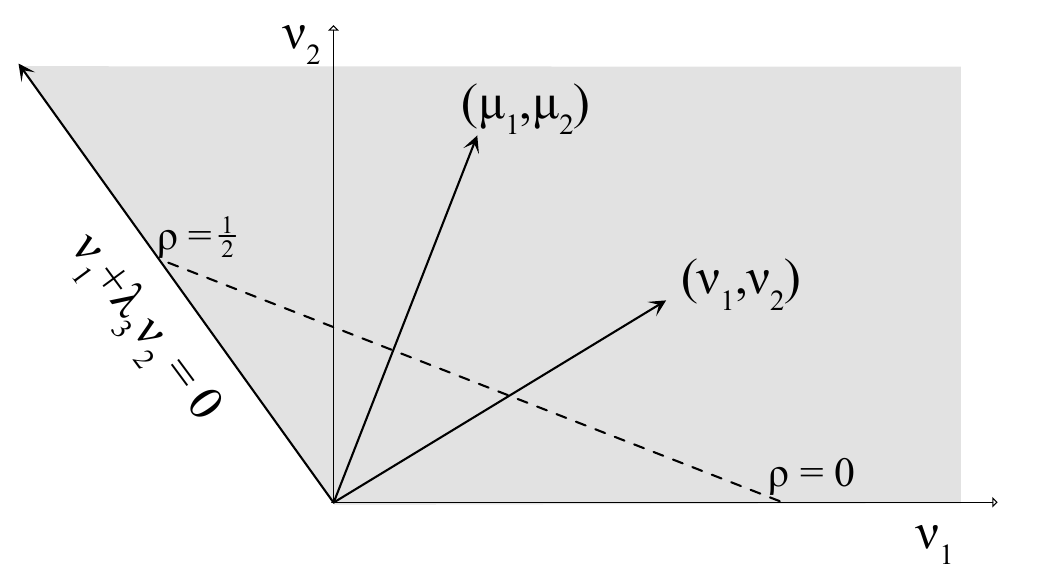}
{Shaded region of the parameter plane $(\nu_1,\nu_2)$ in which the pencil \Eq{PencilSet}
obeys conditions \Eq{nuRestrictions} so that
$\cC_\bnu$ is inside $\cC_1$. The representative vectors correspond to $\cC_\bmu$ being inside $\cC_\bnu$ according to \Lem{restrict}.
The dashed line shows the variation of the rotation number along the segment \Eq{LineSegment}.}
{PencilRegion}{3.8in}

\subsection{Pencil Eccentricity}\label{sec:PencilE}

For a pair of ellipses satisfying the standing assumptions,
the \emph{pencil eccentricity} \Eq{PencilEccentricity} is an invariant that
depends only on the eigenvalues $\lambda_1,\lambda_2,\lambda_3$ of
$C_1^{-1}C_2$. As we will see below, this quantity
determines the rotation number of the associated Poncelet map.
First we show that the eccentricity is pencil invariant:
given an outer ellipse the eccentricity is the same for \emph{any}
interior ellipse of the pencil.

\begin{lem}\label{lem:ModifiedEccentricity}
Suppose that $C_1$ and $C_2$ satisfy \Ass{SA1}-\ref{ass:SA3} and that $\cC_\bnu$ is the associated pencil \Eq{PencilSet}.
Then if $\bnu$ and $\bmu$ satisfy \Eq{nuRestrictions},
the eccentricity of the pencil is invariant:
\[
	e(C_1,C_\bnu)=e(C_1,C_\bmu)=e(C_1,C_2) \;.
\]
More generally, if $\cC_\bmu$ is inside $\cC_\bnu$, then
\beq{ModifiedEccentricity}
	e(C_\bnu,C_\bmu) = \sqrt{\frac{\nu_1+\lambda_3\nu_2}{\nu_1+\lambda_2\nu_2}}\,e(C_1,C_2).
\eeq
\end{lem}

\begin{proof}
The eigenvalues of $C_1^{-1}C_\bnu$ are precisely $\widetilde{\lambda}_i=\nu_1+\lambda_i\nu_2$, $i=1,2,3$. This implies that
\[
	e(C_1,C_\bnu)=
 		\sqrt{\frac{\widetilde{\lambda}_1-\widetilde{\lambda}_2}{\widetilde{\lambda}_1-\widetilde{\lambda}_3}}=
 		\sqrt{\frac{\lambda_1-\lambda_2}{\lambda_1-\lambda_3}}= e(C_1,C_2).
\]
For the more general case, note that the eigenvalues of the matrix $C_\bnu^{-1}C_\bmu$ are
\[
	\bar{\lambda_i} = \frac{\mu_1 + \lambda_i \mu_2}{\nu_1 + \lambda_i \nu_2},
\]
thus
\[
	e(C_\bnu,C_\bmu) = \sqrt{\frac{\bar\lambda_1-\bar\lambda_2}{\bar\lambda_1-\bar\lambda_3}}
	 = \sqrt{\frac{\nu_1+\lambda_3\nu_2}{\nu_1+\lambda_2\nu_2}}
	 	 \sqrt{\frac{(\nu_1\mu_2-\nu_2\mu_1)(\lambda_1-\lambda_2)}{(\nu_1\mu_2-\nu_2\mu_1)(\lambda_1-\lambda_3)}},
\]
which reduces to the promised expression \Eq{ModifiedEccentricity}.
\end{proof}

\begin{remark}
Note that, as in \Rem{MatrixScaling}, the eccentricity \Eq{PencilEccentricity} is also invariant under
the scaling transformation:
$
	e( C_1, C_2)=e(\delta C_1,\gamma C_2),
$
for $\delta,\gamma>0$.
Furthermore, whenever $C_1$ and $C_2$ satisfy \Ass{SA2}, then so do the pair of congruent matrices
$\widetilde{C}_1=S^TC_1S$ and $\widetilde{C}_2=S^TC_1S$,
when $S$ is invertible. Note that, since
\(
	\widetilde{C}_1^{-1}\widetilde{C}_2 = S^{-1}(C_1^{-1}C_2)S ,
\)
the eigenvalues for \Ass{SA2} are unchanged.
Thus, the eccentricity is invariant under such a transformation:
$
	e( C_1, C_2)=e(\widetilde{C}_1,\widetilde{C}_2).
$
\end{remark}

\subsection{Projective Transformations}\label{sec:Projective}

In this section we will use the matrix $M$ that satisfies
\Eq{SylvesterM} to create a projective transformation of the plane.
Recall that under a projective transformation, lines are
transformed into lines, and quadrics into quadrics. Thus
such a transformation preserves collinearity and tangency between two curves.
For application to Poncelet maps, we will also require that the projective transformation
be an orientation preserving diffeomorphism.

\begin{defn}[Standard Projective Transformation]
A $3\times 3$ matrix $M$ defines a 
map of the plane
\beq{mhat}
	\widehat M \begin{pmatrix} x \\ y \end{pmatrix}
	=\begin{pmatrix}
 		\displaystyle \frac{m_{11}x+m_{12}y+m_{13}}{m_{31}x+m_{32}y+m_{33}} \\
		\\
 		\displaystyle \frac{m_{21}x+m_{22}y+m_{23}}{m_{31}x+m_{32}y+m_{33}} \\
 	\end{pmatrix}.
\eeq
\end{defn}

\begin{lem}\label{lem:den}
If $M$ is a matrix giving \Eq{SylvesterM}, then $\widehat{M}: \bD \to \widehat{M}(\bD)$ is a diffeomorphism, where
$\bD$ is the unit disk
\beq{UnitDisk}
	\bD=\{(x,y)\in\bR^2: x^2+y^2\leq1\}.
\eeq

\end{lem}

\begin{proof}
First we show that whenever $(x,y) \in \bD$,
\[
	m_{31} x + m_{32} y + m_{33} \neq 0.
\]
By way of contradiction, suppose that there exists a
point $(x,y)\in\bD$ and $u,v\in \bR$, such that $M(x,y,1)^T = (u,v,0)^T$.
Since $M$ is nonsingular, $(u,v) \neq (0,0)$, and thus	
$(u,v,0) C_1 (u,v,0)^T > 0$  by \Ass{SA1}.
Using \Eq{SylvesterM}, we find that
\[
	x^2 + y^2 -1= (x,y,1)M^T C_1 M (x,y,1)^T = (u,v,0) C_1 (u,v,0)^T > 0,
\]
contradicting $(x,y)\in\bD$. This gives the result.
Finally, it is easy to see that $\widehat{M}$ is one-to-one. Indeed, the inverse of the transformation
\Eq{mhat} is simply obtained by replacing the nonsingular matrix $M$ by its inverse.
\end{proof}


\begin{lem}\label{lem:orient}
 The map \Eq{mhat} is orientation-preserving if and only if $\det(M)\,m_{33}>0$.
\end{lem}

\begin{proof}
By \Lem{den} the denominator $m_{31}x+m_{32}y+m_{33}\neq0$, for all $(x,y)\in \bD$.
In particular, $m_{33}\neq0$ and $m_{31}x+m_{32}y+m_{33}$ does not change sign on $\bD$.
The determinant of the Jacobian of $\widehat{M}$ is
\[
 	\det\left(D\widehat{M}(x,y)\right)=\frac{\det(M)}{(m_{31}x+m_{32}y+m_{33})^3}.
\]
Since the sign of the denominator above is the sign of $m_{33}$, we conclude that
$ \det\left(D\widehat{M}(x,y)\right)>0$, for all $(x,y)\in D$, if and only if
$\det(M)$ and $m_{33}$ have the same sign.
\end{proof}

\begin{remark}
The matrix $M$ defined by conditions \Eq{SylvesterM} is not unique. If $M$ satisfies \Eq{SylvesterM},
then any of the eight matrices,
\[
 	M\,\diag(\pm 1, \pm 1, \pm 1),
\]
will also satisfy equations \Eq{SylvesterM}.
Hence, we can modify the matrix $M$ and assume---without loss of generality---that $\det(M)\,m_{33}>0$,
so that $\widehat{M}$ is orientation-preserving.
\end{remark}

The projective transformation \Eq{mhat} can be applied to any pencil \Eq{PencilSet} under
the restrictions \Eq{nuRestrictions} to give the simple form.
\beq{util1}
 	\overline\cC_\bnu \equiv \widehat M ^{-1}(\cC_\bnu)
		= \left\{(X,Y) \in \bR^2 : (\nu_1 +\lambda_1\nu_2) X^2 + (\nu_1 + \lambda_2\nu_2)Y^2
		= \nu_1 + \lambda_3\nu_2 \right\}.
\eeq
Since $\widehat{M}$ is an orientation preserving diffeomorphism, the interior of $\overline\cC_\bnu$ becomes
\[
\interior(\overline\cC_\bnu)=\widehat M ^{-1}(\interior(\cC_\bnu)) = \left\{(X,Y) \in \bR^2 : (\nu_1 +\lambda_1\nu_2) X^2 + (\nu_1 + \lambda_2\nu_2)Y^2
		< \nu_1 + \lambda_3\nu_2 \right\}.
\]
Note that in these coordinates the outer ellipse, where $\nu_2=0$, simply becomes the unit circle
\[
	 \bS^1 = \widehat M ^{-1}(\cC_{(1,0)})
	 	= \left\{(X,Y) \in \bR^2 : X^2 + Y^2 = 1 \right\}.
\]
and the inner ellipse, where $\nu_1 = 0$, becomes
\beq{StdInnerEllipse}
	 \overline\cC_i = \widehat M ^{-1}(\cC_{(0,1)})
	 		= \left\{(X,Y) \in \bR^2 : \lambda_1 X^2 + \lambda_2 Y^2 = \lambda_3\right\}.
\eeq
The pencil \Eq{util1} also contains the limiting ellipse, on the boundary $\nu_1+\lambda_3\nu_2= 0$
of the region \Eq{nuRestrictions}. Choosing $\bnu = (-\lambda_3,1)$ to represent this case,
then the corresponding transformed conic is simply the origin:
\[
	\overline \cC_{(-\lambda_3,1)}= \{(0,0)\}.
\]
Of course this point is in the interior of all the ellipses in the pencil $\overline\cC_\bnu$. As a consequence,
the limiting point for the pencil satisfies
\beq{LimitingPoint}
	\cC_\infty \equiv \cC_{(-\lambda_3,1)}= \{\widehat M (0,0) \}.
\eeq

Finally, under the assumption of \Lem{orient}, the Poncelet map $P:\cC_o\to \cC_o$, with respect to the interior ellipse $\cC_i$, is transformed, by the conjugacy
\[
	\overline{P}= \widehat{M}^{-1}\circ P\circ \widehat{M}.
\]
to a Poncelet map $\overline{P} : \bS^1\to \bS^1$ with respect to
the inner ellipse \Eq{StdInnerEllipse}.

\section{Poncelet maps for Pencils}\label{sec:PonceletMaps}

In \Sec{StandardPencil} we will show that the general pencil \Eq{PencilSet}, with eccentricity $e$
given by \Eq{PencilEccentricity},
can be mapped onto a \emph{standard pencil} that we call $\cS_e^f$ so that the outermost ellipse is the unit circle.
We see in \Sec{Parameterization} that this transformation conjugates \Eq{PencilPoncelet}
to the standard Poncelet map $P_e^f: \bS^1 \to \bS^1$ for the standard pencil.
We will then find a second conjugacy that maps any pair of confocal ellipses \Eq{Ellipse} onto the standard
pencil. The implication is that the billiard map $B_e^f$ is also conjugate to $P_e^f$. Thus by this two-step
process, we find a conjugacy between $B_e^f$ and $P_\bnu$. An implication is that the rotation numbers of $B_e^f$, $P_e^f$, and $P_\bnu$ are all equal to $\rhoin(e,f)$.

In \Sec{ExplicitRotation} we obtain the relation \Eq{PencilRotation}
between $\bnu$ and the rotation number of the Poncelet map.
We then prove \Th{PencilParam}, giving the inverse of this relation.
Finally in \Sec{Monotonicity}, we prove that the rotation number
satisfies the monotonicity condition \Th{Monotonicity}.

\subsection{Standard Pencil}\label{sec:StandardPencil}

A \emph{standard pencil} is a family of ellipses with a common center for which the outermost is a circle:

\begin{defn}[Standard Pencil]
The \emph{standard pencil} with eccentricity $e \in (0,1)$ is the one-parameter family of ellipses
\beq{StandardPencil}
	\cS_e^f=\left\{(X,Y)\in\bR^2:\left(\frac{1-f^2}{1-e^2}\right)\,X^2+Y^2=\frac{f^2}{e^2}\right\},
\eeq
for  $0\leq f\leq e$. The outermost conic $\cS_e^e = \bS^1$ is
the unit circle, and the innermost is the limiting point $\cS_e^0=\{(0,0)\}$.
\end{defn}

\noindent
Note that the ellipses in the set \Eq{StandardPencil} are nested:  whenever $f_1<f_2$ then
$\cS_e^{f_1}\subset\interior\left(\cS_e^{f_2}\right)$. Moreover,
$
 	\bigcup_{f=0}^e\cS_e^f=\bD,
$
the unit disk \Eq{UnitDisk}.

It is easy to see that $\{\cS_e^f\}_{0\leq f\leq e}$ is a pencil of the form \Eq{PencilSet}.
For the outer circle the associated matrix is $C_1 = \diag(1,1,-1)$ and
for  the ellipse $S_e^{f}$, it is $C_2 =\diag(\lambda_1\lambda_2,-\lambda_3)$ where
\beq{CanonicalLambda}
 	\lambda_1=\frac{1-f^2}{1-e^2},\qquad \lambda_2=1,\qquad\lambda_3=\frac{f^2}{e^2}.
\eeq
Since these matrices are already diagonal, \Eq{CanonicalLambda} are the eigenvalues for \Eq{SylvesterM},
and when $0<f<e$, they satisfy the ordering \Eq{EVOrdering}.
Using \Eq{PencilEccentricity}, the eccentricity of this pencil is therefore $e(C_1,C_2) = e$,
which is---as expected---independent of $f$.

With some algebra, one can show that the standard Poncelet map, $P_e^f: \bS^1 \to \bS^1$,
with respect to the inner ellipse $\cS_e^f$, is explicitly
\beq{PefDefine}
 	P_e^f(X,Y)=
 		\left(-\frac{A_1 A_4\, Y \sqrt{1-e^2 Y^2}+A_2 A_3\, X}{A_2^2-A_1^2 e^2 Y^2},\frac{A_1 A_2\, X \sqrt{1-e^2 Y^2}-A_3 		A_4 Y}{A_2^2-A_1^2 e^2 Y^2}\right),
\eeq
where
\begin{align*}
 	A_1 &= 2 f \sqrt{\left(1-f^2\right) \left(e^2-f^2\right)}, \\
 	A_2 &= e^2-f^4,\\
 	A_3 &=e^2+f^4-2 f^2,\\
 	A_4 &= e^2 \left(1-2 f^2\right)+f^4.
	\end{align*}
Note that when $f = 0$, $P_e^0(x,y)=(-x,-y)$, so that
the segment $\overrightarrow{(x,y)P_e^0(x,y)}$ passes through the limiting point, $S_e^0$,
so that the rotation number is $\tfrac12$.
Moreover, $P_e^e(x,y)=(x,y)$ is the identity map since $S_e^e = \bS^1$.
As we will see, $P_e^f$ is conjugate to $B_e^f$, so it has rotation number $\rhoin(e,f)$.

The standard pencil \Eq{StandardPencil} is mapped onto the general pencil using the transformation $\widehat{M}$ of \Eq{mhat}. Moreover a simple linear transformation maps the family of confocal ellipses \Eq{Ellipse} onto the standard pencil:

\begin{lem}\label{lem:CanonicalMaps}
Given matrices $C_1$ and $C_2$ with eigenvalues $\lambda_1,\lambda_2,\lambda_3$ obeying \Ass{SA1}-\ref{ass:SA3},
let $\cC_\bnu$ be the pencil \Eq{PencilSet} for a valid parameter vector $\bnu=(\nu_1,\nu_2)$, satisfying \Eq{nuRestrictions}.
If $e=e(C_1,C_2)$ is the pencil eccentricity and
\beq{Pencilf}
	f=f(\bnu)=e\sqrt{\frac{\nu_1+\lambda_3\nu_2}{\nu_1+\lambda_2\nu_2}},
 \eeq
 then:
\begin{enumerate}
\item $\widehat{M}\big(\bS^1\big)=\cC_1$ and $\widehat{M}\big(\cS_e^f\big)=\cC_\bnu$ using the transformation \Eq{mhat};
\item The linear transformation
\beq{transf}
	 T_f(x,y)=\left(
 		\frac{f}{\sqrt{1-f^2}}\, y ,
 			-f\,x\right),
\eeq
maps $T_f(E(f))=\bS^1$ and $T_f(E(e))=\cS_e^f$, where $E(\eps)$ is the confocal family
defined in \Eq{Ellipse}.
\end{enumerate}
\end{lem}

\begin{proof}
By construction, the projective transformation $\widehat{M}$ satisfies \Eq{util1}, which implies that
$\widehat{M}^{-1}(\cC_1) = \bS^1$ since $\bnu = (1,0)$. 
Therefore, to prove the first part, it is enough to show that $\cS_e^f = \widehat M^{-1}(\cC_\bnu)$.
From \Eq{PencilEccentricity} and \Eq{Pencilf}, since $0<f<e$, we find that
 \[
 	\frac{1-f^2}{1-e^2}=1+\left(\frac{e^2}{1-e^2}\right)\left(1-\frac{f^2}{e^2}\right)
 				=1+\left(\frac{\lambda_1-\lambda_2}{\lambda_2-\lambda_3}\right)
				 \left(\frac{(\lambda_2-\lambda_3)\nu_2}{\nu_1+\lambda_2\nu_2}\right)
				=\frac{\nu_1+\lambda_1\nu_2}{\nu_1+\lambda_2\nu_2}.
 \]
 This implies that $\cS_e^f$ has the same form as \Eq{util1}.

 For the second part, given any point $(x,y)\in\bR^2$, let $(X,Y)=T_f(x,y)$. A direct computation shows that
 $(x,y)\in E(f)$ if and only if $X^2+Y^2=1$, and
 $(x,y)\in E(e)$ if and only if $(X,Y)\in \cS_e^f$.
\end{proof}

\begin{remark}
We know that the eccentricity of a pencil is invariant under projective transformations. In
fact, as a consequence of \Lem{CanonicalMaps}, any pencil with eccentricity $e$ is
projectively-equivalent to the standard pencil $\{\cS_e^f\big\}_{0\leq f\leq e}$.
Furthermore, two pencils of nested ellipses are projectively-equivalent if and only if they
have the same eccentricity.
\end{remark}

\begin{remark}
Suppose that $0<f<e<1$ and $\ell=\rhoin(e,f)$.
Substituting  $f=e\,\cd(2K(e)\ell,e)$ from \Eq{poristic} into \Eq{StandardPencil}, implies that
\[
	\cS_e^f=\left\{(X,Y)\in\bR^2:X^2+ \dn^2(2K(e)\ell,e)\,Y^2=\cn^2(2K(e)\ell,e)\,\right\}.
\]
This also implies that for a fixed pencil eccentricity $e$, the rotation number $\ell$ can be used to
parameterize the standard pencil.
\end{remark}

\subsection{Standard Covering Space}\label{sec:Parameterization}

Using \Lem{reduction} and \Lem{CanonicalMaps}, we will now obtain
a general covering space that will simultaneously reduce all Poncelet maps,
and make them conjugate to rigid rotation.

\begin{defn}[Standard Covering Map]\label{def:StdCover}
Let $C_1$ and $C_2$ be a pair of matrices  obeying \Ass{SA1}-\ref{ass:SA3}.
If $M$  diagonalizes $C_1$ and $C_2$ as in \Eq{SylvesterM}, and $\det(M)m_{33}>0$, then
the standard covering map $\Pi_M:\bR\to\cC_1$, is
\beq{stdcover}
	\Pi_M(\theta) = \widehat{M} \begin{pmatrix} \cn(4K(e)\theta,e) \\ \sn(4K(e)\theta,e) \end{pmatrix} ,
\eeq
where $\widehat{M}$ is given in \Eq{mhat} and $e$ is the pencil eccentricity \Eq{PencilEccentricity}.
\end{defn}

\begin{remark}
  The parameterization \Eq{stdcover} is well defined since \Lem{den} implies that
the denominator in these terms can never be zero because $(\sn(u,e), \cn(u,e)) \in \bD$.
Moreover, the periodicity of the Jacobi elliptic functions implies that
$\Pi_M$ is a covering map with group of deck transformations generated by $\theta\mapsto\theta+1$.
The condition $\det(M)m_{33}>0$ is important so that $\widehat{M}$ is orientation-preserving.
In particular, if $M$ is the identity, then $\theta \mapsto  (\cn(4K(e)\theta,e) ,\sn(4K(e)\theta,e))$
is the covering map for the standard pencil.
\end{remark}

\begin{lem}\label{lem:PencilPoncelet}
Suppose that $C_\bnu$ is a pencil \Eq{PencilSet} with eigenvalues $\lambda_1 > \lambda_2 > \lambda_3 > 0$
and pencil eccentricity $e$.
If $P_\bnu : \cC_1 \to \cC_1$ is the family of Poncelet circle maps for the pencil, then $\Pi_M:\bR\to\cC_1$ is a covering map and
\beq{CoverConjugacy}
	P_\bnu\circ \Pi_M = \Pi_M \circ  \widetilde{B}_e^f ,
\eeq
where $f=f(\bnu)$ as in \Eq{Pencilf}.
 \end{lem}

\begin{proof}
Given $\bnu=(\nu_1,\nu_2)$ obeying \Eq{nuRestrictions}, we define $f$ by \Eq{Pencilf}. We
will use $\widehat{M}$ \Eq{mhat} and $T_f$ \Eq{transf} to show that
the lift of $P_\bnu$ is a rigid rotation.
The expression \Eq{maincover} for $\pi_e^f$, gives
\[
 	T_f\circ \pi_e^f(\theta)=\begin{pmatrix}
 		\cn\left(4K(e)\theta,e\right) \\
 		\sn\left(4K(e)\theta,e\right) \\
 		\end{pmatrix}.
\]
Since the covering map \Eq{stdcover} depends only on the pencil eccentricity
$e$, we can use it for any member of the pencil.
Since $\left.T_f\right|_{E(f)}:E(f)\to \bS^1$ and $\left.\widehat{M}\right|_{\bS^1}:\bS^1\to \cC_1$ are
diffeomorphisms and $\pi_e^f:\bR\to E(f)$ is a covering map, then
\[
	\Pi_M=\left.\widehat{M}\right|_{\bS^1}\circ \left.T_f\right|_{E(f)}\circ \pi_e^f,
\]
is a covering map of $\cC_1$.
Let $\overline{P}_\bnu:\bS^1\to \bS^1$ be the circle map given by
$\overline{P}_\bnu(x,y)=(\widehat{M})^{-1}\circ P_\bnu\circ \widehat{M}(x,y)$, for all $(x,y)\in \bS^1$.
Using \Lem{CanonicalMaps} and the fact that $\widehat{M}$ is an orientation-preserving
projective transformation, $\overline{P}_\bnu$ is the Poncelet map with outer ellipse
$\bS^1$ and inner ellipse $(\widehat{M})^{-1}(C_\bnu)= \cS_e^f$.
Thus $\overline{P}_\bnu=P_e^f$, \Eq{PefDefine}.

We notice that the transformation $T_f$ is linear and orientation-preserving. Given that
$E(f)=T_f^{-1}(\bS^1)$ and $E(e)=T_f^{-1}(\cS_e^f)$, we find that
$T_f^{-1}\circ \overline{P}_\bnu\circ T_f$ is the Poncelet map
with outer ellipse $E(f)$ and inner ellipse $E(e)$. The only possibility is that
$T_f^{-1}\circ \overline{P}_\bnu\circ T_f=B_e^f$.

Combining these results with \Th{ConfocalCover} gives the following commutative diagram,
where $e$ is the eccentricity of the pencil and $f=f(\bnu)$ is given by \Eq{Pencilf}.
\[
    \xymatrix @R=3pc @C=3pc {
     {\bR }\ar[r]^{{\widetilde B}_e^f} \ar[d]_{\pi_e^f}\ar@/_2.5pc/[ddd]_{\Pi_M} &
     \ar@/^2.5pc/[ddd]^{\Pi_M} {\bR }\ar[d]^{\pi_e^f}\\
     E(f)\ar[d]_{T_f} \ar[r]^{B_e^f}& E(f)\ar[d]^{T_f} \\
     \bS^1\ar[d]_{\widehat{M}} \ar[r]^{\overline{P}_\bnu=P_e^f} & \bS^1\ar[d]^{\widehat{M}} \\
     \cC_1\ar[r]^{P_\bnu} & \cC_1\\
    }
\]
This implies \Eq{CoverConjugacy}.
\end{proof}

We are finally ready to prove \Th{PencilRotation} from \Sec{Introduction}:

\begin{proof}[Proof of \Th{PencilRotation} (Pencil Rotation Number)]
Lemma~\ref{lem:PencilPoncelet} implies that if $\bnu$ is valid \Eq{nuRestrictions}, then
$\widetilde{B}_e^f(\theta)=\theta+\rhoin(e,f)$, is a lift of $P_\bnu$ under the cover $\Pi_M$.
Recalling the billiard rotation number $\rhoin$ and $\omegain$ from \Eq{BilliardRotationNumber}, we can substitute
$f$ from \Eq{Pencilf} to obtain $\omega(\bnu)=\omegain(e,f(\bnu))$, which gives \Eq{PencilRotation}.
This concludes the proof of \Th{PencilRotation}.
\end{proof}

\begin{remark}
  If we write $B_\bnu=\widetilde{B}_e^{f(\bnu)}$, we get the standard
  commutative diagram in \Sec{IntroStdCovering}. In this case
  $B_\bnu(\theta) = \theta+\hat{\rho}(e,f(\bnu)) = \theta + \rho(\bnu)$.
\end{remark}

\begin{cor}\label{cor:EquivariantSym}
There exists an equivariant symmetry $\Phi_t: \cC_1 \to \cC_1$ such that
	$P_{\bnu}\circ \Phi_t=\Phi_t\circ P_{\bnu}$, for all $t \in \bR$ and all valid $\bnu$.
\end{cor}

\begin{proof}
For each $t$, we can define $\Phi_t: \cC_1 \to \cC_1$ in terms of the covering map  \Eq{stdcover} as
\[
	\Phi_t(\Pi_M(\theta))=\Pi_M(\theta+t).
\]
Given an element $z_0\in\cC_1$, there exists a $\theta_0\in\bR$ such that $z_0=\Pi_M(\theta_0)$. Then
$\Phi_t(z_0)=\Pi_M(\theta_0+t)$ and $P_{\bnu}( z_0)=\Pi_M(\theta_0+\rho(\bnu))$ and hence
\[
	P_{\bnu}( \Phi_t(z_0))=P_{\bnu}\left(\Pi_M(\theta+t)\right)=\Pi_M(\theta_0+t+\rho(\bnu))
		=\Phi_t(P_{\bnu}( z_0)),
\] for all $t \in \bR$.
\end{proof}

\subsection{Explicit Rotation Number and Inverse Parameter Conditions}\label{sec:ExplicitRotation}
In this section we will prove \Th{PencilParam}, thus showing how to solve the inverse problem:
given a rotation number, find an element of a pencil whose Poncelet map
has that rotation number.
As usual we let $C_\bnu$ be a pencil \Eq{PencilSet} of nested ellipses and
assume that $\bnu$ satisfies \Eq{nuRestrictions}.

\begin{proof}[Proof of \Th{PencilParam} (Inverse Parameter Conditions)]
 By \Lem{PencilPoncelet}, the Poncelet map $P_\bnu:\cC_1 \to \cC_1 $
 is conjugate to $B_e^f$, where $f$ is given by \Eq{Pencilf}. The rotation
 number for the billiard map is $\rhoin(e,f)=\ell \in (0,\tfrac12)$
 if and only if $f=e\,\cd(2K(e)\ell,e)$, by \Lem{porism}.
 Thus $P_\bnu$ has rotation number $\ell$ if and only if
 \beq{CanonicalForm}
	\cd^2 (2K(e)\ell,e) =\frac{f^2}{e^2}=\frac{\nu_1+\lambda_3\nu_2}{\nu_1+\lambda_2\nu_2}.
 \eeq
Solving for $\nu_1/\nu_2$, and using standard Jacobi elliptic function identities then
gives the promised ratio \Eq{InverseRotationMap}.
Since the $\cC_\bnu$ depends only on the ratio $\nu_1/\nu_2$, we can, choose the parameters as
\bsplit{ParamCanonical}
	 \nu_1(\ell)&=\lambda_1\cn^2(2K(e)\ell,e)-\lambda_3, \\
	  \nu_2(\ell)&=\sn^2(2K(e)\ell,e) .
\esplit
These parameters trace the line segment
\beq{LineSegment}
	\nu_1(\ell)+\lambda_1 \nu_2(\ell)=\lambda_1-\lambda_3
\eeq
that was sketched in \Fig{PencilRegion}.
Note that $\nu_2(\ell)>0$ and $\nu_1(\ell)+\lambda_3 \nu_2(\ell)=(\lambda_1-\lambda_3)\cn^2(2K(e)\ell,e)>0$.
Then \Lem{interior} implies that each $\cC_{\bnu(\ell)}$ is an interior ellipse of the pencil.
Thus $P_{\bnu(\ell)}$ is conjugate to rigid rotation using the standard covering in \Lem{PencilPoncelet}.
\end{proof}

This result is similar to the explicit porism given in \Lem{porism}
for the billiard map, which determined
the eccentricity $f$ of the outer ellipse for a given inner
ellipse $e$ and rotation number $\ell$.
As a simple example, we generalize the case $\ell = \tfrac14$ from \Ex{OneFourth}:

\begin{example}[Rotation number $\tfrac14$]
Suppose that $C_\bnu$ is a pencil of nested ellipses \Eq{PencilSet} with eigenvalues
$\lambda_1>\lambda_2>\lambda_3>0$.
Then the Poncelet map $P_{\bnu^*}:\cC_1\to\cC_1$ with
\beq{nuStar}
	\bnu^* = \left(\sqrt{(\lambda_1-\lambda_3)(\lambda_2-\lambda_3)}-\lambda_3, 1 \right)
\eeq
has rotation number $\rho(\bnu^*)=\tfrac14$.
We verify this using  \Eq{PencilRotation}.
Substituting $\bnu^*$ into $\omega(\bnu)$ gives
\[
 	\omega({\bnu}^*)
 		=\arcsin\sqrt{\frac{\left(\lambda _1-\lambda _3\right)}{(\lambda_1-\lambda_3)
		 +\sqrt{(\lambda_1-\lambda_3)(\lambda_2-\lambda_3)}}}
	  =\arcsin\left(\frac{1}{\sqrt{1+e'}}\right) ,
\]
where we have used the eccentricity \Eq{PencilEccentricity}, and defined $e'=\sqrt{1-e^2}$.
This is identical to the $\omega$ computed in \Ex{OneFourth}, which gives \Eq{raro} and thus, by \Eq{PencilRotation}, $\rho({\bnu^*})=\tfrac14$.

Alternatively, a direct use of \Eq{ParamCanonical} is possible. Using
\cite[Form. 122.10]{Byrd1971}, we find that
\[
 \nu_1(1/4)=\lambda_1\cn^2(\tfrac12K(e),e)-\lambda_3=\lambda_1\left(\frac{e'}{1+e'}\right)-\lambda_3,
\]
and
\[
\nu_2(1/4)=\sn^2(\tfrac12K(e),e)=\frac{1}{1+e'} .
\]
A simple computation shows that $\bnu^*=(1+e')\big( \nu_1(1/4), \nu_2(1/4)\big)$, so  $\rho({\bnu^*})=\tfrac14$.
\end{example}

Given a valid parameter $\bnu$, satisfying \Eq{nuRestrictions}, and a corresponding ellipse $\cC_\bnu$,
we can define a \emph{dual} parameter $\overline{\bnu}$ so that the Poncelet map for the new ellipse
$\cC_{\overline{\bnu}}$ has a rotation number that ``complements'' the old rotation number.
In particular, this is useful if we want to identify new poristic ellipses.

\begin{thm}[Dual Parameters] \label{thm:DualParam}
Assume that $\bnu=(\nu_1,\nu_2)$ satisfies \Eq{nuRestrictions},
and define the \emph{dual parameter vector}
\beq{ComplementParams}
	\overline{\bnu}= \big(-\lambda_3 \nu_1+ (\lambda_1 \lambda_2-\lambda_1 \lambda_3-\lambda_2 \lambda_3)\nu_2,
 		\,\nu_1+\lambda_3 \nu_2\big).
\eeq
Then $ \overline{\bnu}$ satisfies \Eq{nuRestrictions}, and $\overline{\overline{\bnu}} =\delta \bnu $, where $\delta=(\lambda_1- \lambda_3)(\lambda_2- \lambda_3)>0$.
Moreover, if the Poncelet maps $P_\bnu$ and $P_{\overline{\bnu}}$ have rotation numbers $\rho(\bnu)$
and $\rho(\overline{\bnu})$, respectively, then
\[
	\rho(\bnu)+ \rho(\overline{\bnu})=\tfrac{1}{2}.
\]
Thus in particular, $\cC_\bnu$ is poristic if and only if $\cC_{\overline{\bnu}}$ is poristic.
\end{thm}
\begin{proof} The first two assertions are trivial.
From the proof of \Th{PencilParam},
$\rho(\overline{\bnu})=\ell'$ if and only if \Eq{CanonicalForm} holds for the substitutions $\ell \to \ell'$, $\bnu\to\overline{\bnu}$.
Substitution of \Eq{ComplementParams} into \Eq{CanonicalForm} and simplification then gives
\[
	\cd^2 (2K(e)\ell',e) = \frac{\overline{\nu}_1+\lambda_3\overline{\nu}_2}{\overline{\nu}_1+\lambda_2\overline{\nu}_2}=
	\frac{\lambda_1-\lambda_3}{\nu_1/\nu_2+\lambda_1}.
\]
If we assume that $\ell'=\tfrac12-\ell$, then \cite[Form. 122.03]{Byrd1971} gives
\[
	\cd \left(2K(e)(\tfrac12-\ell),e\right) =\sn(2K(e)\ell,e).
\]
Thus we have shown that
\[
	\sn^2(2K(e)\ell,e) = \frac{\lambda_1-\lambda_3}{\nu_1/\nu_2+\lambda_1} \quad \Rightarrow
	\quad \frac{\nu_1}{\nu_2} = \frac{\lambda_1(1-\sn^2(2K(e)\ell,e)) -\lambda_3}{\sn^2(2K(e)\ell,e)},
\]
which is equivalent to \Eq{InverseRotationMap}.  
 Thus $\rho(\bnu) = \ell$, and hence
$\rho(\bnu) + \rho(\overline\bnu) = \ell + \tfrac12 - \ell = \tfrac12$, as we wanted.\footnote
{
    Alternatively, it is possible to show that
    \(
    	\cot(\omega(\overline{\bnu}))=\sqrt{1-e^2}\tan(\omega(\bnu)).
    \)
    Using \cite[Form. 117.01]{Byrd1971}, we find that
    \(
    	F(\omega(\bnu),e)+F(\omega(\overline{\bnu}),e)=K(e),
    \)
    and therefore $\displaystyle \rho(\bnu)+ \rho(\overline{\bnu})=\tfrac{1}{2}.$
}

\end{proof}

\noindent

To finish this section, we show that it is possible to identify poristic ellipses using
eigenvalues. Suppose that $(C_o,C_i)$ be a pair of matrices satisfying \Ass{SA1}-\ref{ass:SA3} and
that $\lambda_1>\lambda_2>\lambda_3>0$ are the corresponding eigenvalues.
If we define the parameters
\[
\bnu_o=(1,0),\quad \text{and}\quad \bnu_i=(0,1).
\]
then $\bnu_i$ is valid according to \Eq{nuRestrictions} and the ellipses become
$\cC_o=\cC_{\bnu_o}$ and $\cC_i=\cC_{\bnu_i}$.
The dual to $\bnu_i$, \Eq{ComplementParams} is then
 \[
  	\overline{\bnu}_i=\big(\lambda_1 \lambda_2-\lambda_1 \lambda_3-\lambda_2 \lambda_3,
 		\lambda_3\big).
 \]
This implies that the pair $(C_o,C_{\overline{\bnu}_i})$ also satisfies \Ass{SA1}-\ref{ass:SA3}, and
has the eigenvalues
\beq{explicitCompl}
	\overline{\lambda}_1=\lambda_2(\lambda_1-\lambda_3),\qquad
 	\overline{\lambda}_2=\lambda_1(\lambda_2-\lambda_3),\qquad
 	\overline{\lambda}_3=(\lambda_1-\lambda_3)(\lambda_2-\lambda_3).
\eeq
 In particular, $ \overline{\lambda}_1> \overline{\lambda}_2> \overline{\lambda}_3>0$.
Here are several examples using this formula.

\begin{thm}\label{thm:LambdaConditions}
Let $(C_o,C_i)$ be defined as above so that $C_o^{-1}C_i$ has eigenvalues $\lambda_1 > \lambda_2 > \lambda_3 > 0$. 
If  $P_{\bnu_i}$ is the corresponding Poncelet map for $\cC_{\bnu_i}$, then we have the following poristic cases.
\begin{enumerate}
    \item $\rho(\bnu_i)=\tfrac14$ if and only if
    \beq{OneFourth2}
    	\lambda_1 \lambda_2-\lambda_1 \lambda_3-\lambda_2 \lambda_3=0.
    \eeq
    \item $\rho(\bnu_i)=\tfrac13$ if and only if
    \beq{OneThird2}
    	\frac{1}{\sqrt{\lambda_3}}=
    	\frac{1}{\sqrt{\lambda_1}}+
    	\frac{1}{\sqrt{\lambda_2}}.
    \eeq
    \item $\rho(\bnu_i)=\tfrac16$ if and only if
    \beq{OneSixth2}
    	\sqrt{1-\frac{\lambda_3}{\lambda_1}}+
    	\sqrt{1-\frac{\lambda_3}{\lambda_2}}=1.
    \eeq
\end{enumerate}
\end{thm}

\begin{proof}
For the first part, we notice that $\rho(\bnu_i)=\tfrac14$ if and only if
$\rho(\overline{\bnu}_i)=\tfrac14$.   This
happens if and only if  ${\bnu}_i$ and $\overline{\bnu}_i$ are parallel, and hence $\lambda_1 \lambda_2-\lambda_1 \lambda_3-\lambda_2 \lambda_3=0$.
Equivalently, $\bnu_i=(0,1)$ is parallel to $\bnu^*$ in \Eq{nuStar}.

Let $e$ be the eccentricity of the pencil \Eq{PencilEccentricity} and $f=e\sqrt{\lambda_3/\lambda_2}$. We know
that $\rho(\bnu_i)=\tfrac13$ when $e$ and $f$ solve the polynomial \Eq{OneThird}.
Upon rearrangement this gives $(1-f^2)(e-f)^2=f^2(1-e^2)$, or equivalently
\[
	 \left(\frac{e}{f}-1\right)^2= \frac{ 1-e^2}{ 1-f^2}.
\]
Substituting for  $e$ and $f$  then
implies that $\rho(\bnu_i)=\tfrac13$ if and only if
\[
	\left(\sqrt{\frac{\lambda_2}{\lambda_3}}-1\right)^2=\frac{\lambda_2}{\lambda_1},
\]
which gives \Eq{OneThird2} upon rearrangement.
Finally, using the dual parameters \Eq{ComplementParams}, we know that  $\rho(\bnu_i)=\tfrac16$ if and only if $\rho(\overline{\bnu}_i)=\tfrac13$.
Noting that $(C_o,C_{\overline{\bnu}_i})$ has the eigenvalues \Eq{explicitCompl} and using these
in \Eq{OneThird2} gives the result \Eq{OneSixth2}.
\end{proof}

\begin{remark}
  Suppose that the generalized characteristic polynomial of the pencil is of the form
  \[
  	\det(\lambda C_1-C_2)=
   		\det( C_1)(\lambda-\lambda_1)(\lambda-\lambda_2)(\lambda-\lambda_3)
   		=\beta_0 \lambda^3+\beta_1 \lambda^2+\beta_2 \lambda+\beta_3.
  \]
  As discussed in \Sec{IntroPorism}, the Cayley condition is computed from the expansion
  of the square root of this polynomial around $\lambda = 0$, \Eq{CayleyExp}.
  For this case, the Cayley condition becomes
  \[
  		\cay_3  = \alpha_2 = \frac{\sqrt{\beta_3}}{8}\left(4\left(\frac{\beta_1}{\beta_3}\right)
			-\left(\frac{\beta_2}{\beta_3}\right)^2\right).
  \]
  Thus $\rho(\bnu_i)=\tfrac13$ if and only if $\alpha_2 = 0$, or $\beta^2_2=4\beta_1\beta_3$.
  Given that $\lambda_1>\lambda_2>\lambda_3>0$, it is not hard to see that this
  corresponds exactly with the condition that we found in the theorem above.
 \end{remark}

\begin{example}[Limiting cases]
The outer ellipse  of the pencil \Eq{PencilSet} corresponds to the parameter $\bnu_o = (1,0)$. The dual of
this point, using the definition \Eq{ComplementParams}, is $\overline{\bnu}_o=(-\lambda_3,1)$,
on the boundary of \Eq{nuRestrictions}. This corresponds to the limiting point
$\cC_{\overline{\bnu}_o} = \cC_\infty$ \Eq{LimitingPoint}.
Geometrically, it is clear that
\[
	 \rho(\bnu_o)= \rho(1,0)=0, \qquad \rho(\overline{\bnu}_o)=\rho(-\lambda_3,1)=\tfrac{1}{2}.
\]
This verifies that $\rho(\bnu_o)+\rho(\overline{\bnu}_o)=\frac{1}{2}$.
\end{example}

\begin{example}[Euler and Fuss bicentric conditions]
When the inner and outer ellipses are circles and the Poncelet map is poristic,
the resulting orbit is a ``bicentric polygon''. The cases of bicentric triangles and quadrilaterals
were famously studied by Chapple, Euler, and Fuss \cite{Dragovic11, Nohara12,Cieslak18}.
Here we show that the conditions of \Th{LambdaConditions} generalize these results.
Consider the circles
\begin{align*}
	\cC_1 &=\left\{(x,y)\in\bR^2:x^2+y^2=R^2\right\}, \\
	\cC_2&=\left\{(x,y)\in\bR^2:(x-a)^2+y^2=r^2\right\},
\end{align*}
where $a,r,R>0$.
Clearly $\cC_2\subset \interior(\cC_1)$ when $R > r + |a|$.
In particular, this implies that $|a^2-r^2|<R^2$.
These circles can be represented by the pair of matrices
\[
      C_1= \begin{pmatrix}
         1 & 0 & 0 \\
         0 & 1 & 0 \\
         0 & 0 & -R^2 \\
    \end{pmatrix},\qquad
    C_2=\begin{pmatrix}
         1 & 0 & -a \\
         0 & 1 & 0 \\
         -a & 0 & a^2-r^2 \\
    \end{pmatrix}.
\]
The characteristic polynomial $ \det(\lambda C_1-C_2)$ has the eigenvalues
\[
	\lambda_1=1,\quad \lambda_{2,3}=\frac{b\pm\sqrt{b^2-4 r^2 R^2}}{2 R^2},
\]
where $b=R^2+r^2-a^2>0$. It is easy to see that $\lambda_1>\lambda_2>\lambda_3>0$.

Using \Th{LambdaConditions} we can find the relations between $a,r$ and $R$ so that
the rotation numbers are $\tfrac14$ and $\tfrac13$.
These correspond to bicentric triangles and quadrilaterals, respectively.
In particular, \Eq{OneFourth2} gives
\[
    \lambda_1 \lambda_2-\lambda_1 \lambda_3-\lambda_2 \lambda_3=0=\frac{\sqrt{b^2-4 r^2 R^2}-r^2}{R^2}.
\]
With some algebra, we find that this holds if and only if
\[
	\frac{1}{(R-a)^2}+\frac{1}{(R+a)^2}=\frac{1}{r^2},
\]
which is precisely the conditions of Fuss.

For rotation number $\tfrac13$, \Eq{OneThird2} becomes
 \[
   \left(\frac{1}{\sqrt{\lambda_3}}-
    \frac{1}{\sqrt{\lambda_2}}\right)^2-    \frac{1}{\lambda_1}=0=\frac{R (R-2 r)-a^2}{r^2}.
    \]
 With some algebra, we find that this holds if and only if
\[
	\frac{1}{R-a}+\frac{1}{R+a}=\frac{1}{r}.
\]
This is the condition of Chapple-Euler.
\end{example}
\begin{example}
In this example, we consider a pencil of confocal ellipses generated using 
the two ellipses
\[
	\mathcal{Q}_0 = \left\{(x,y)\in\bR^2:\frac{x^2}{a^2}+\frac{y^2}{b^2}=1\right\},
	\qquad
	\mathcal{Q}_{\lambda_0} = \left\{(x,y)\in\bR^2:\frac{x^2}{a^2-\lambda_0^2}+\frac{y^2}{b^2-\lambda_0^2}=1\right\},
\]
where $a>b>\lambda_0>0$.
These ellipses have foci at $(\pm\sqrt{a^2-b^2},0)$, and 
the conditions on the parameters imply that $\mathcal{Q}_{\lambda_0} $ is in the interior of $\mathcal{Q}_{0}$.
Using the notation of \Sec{Standing}, we can represent these using the  matrices
\[
	C_1=\diag\big(a^{-2}, b^{-2},-1\big), \qquad 
	C_2=\diag\big((a^2-\lambda_0^2)^{-1},(b^2-\lambda_0^2)^{-1},-1\big).
\]
Since these are diagonal the three eigenvalues of the pencil are easy to find:
\[
	\lambda_1=\frac{b^2}{b^2-\lambda_0^2}, \qquad 
	\lambda_2=\frac{a^2}{a^2-\lambda_0^2}, \qquad 
	\lambda_3=1,
\]
where, as usual, we have chosen the labels so that $\lambda_1>\lambda_2>\lambda_3>0$. 
For this case the eccentricity $e$ of the pencil is
\[
	e=\sqrt{\frac{\lambda_1-\lambda_2}{\lambda_1-\lambda_3}}=\sqrt{\frac{a^2-b^2}{a^2-\lambda_0^2}}.
\]

Thus the rotation number for the pencil generated by $\mathcal{Q}_0$ and $\mathcal{Q}_{\lambda_0}$ 
is given by \Eq{PencilRotation} with this eccentricity and eigenvalues.
For the case $\bnu_i=(0,1)$, where the inner ellipse is $\mathcal{Q}_{\lambda_0}$, this gives
\[
	\rho(\bnu_i)=\frac{F(\arcsin(\lambda_0/b),e)}{2 K(e)}.
\]
This formula for the rotation number appears in several places in the literature, e.g., \cite{Chang88,Ramirez14,Ramirez17} and \cite[Eq. (8)]{Kaloshin18}.

Using the  formulas of \Th{LambdaConditions}, we can easily identify the values of $\lambda_0$ for 
for which the corresponding Poncelet map has rotation number $1/3, 1/4$, or $1/6$.
For example, \Eq{OneFourth2} implies that $\rho(\bnu_i) = 1/4$  if and only if
\(
	\lambda_0=\frac{ab}{\sqrt{a^2+b^2}}.
\)
Similarly, \Eq{OneSixth2} implies that the rotation number is $1/6$ if and only if 
\(
	\lambda_0=\frac{ab}{a+b}.
\)
These conditions of course match those originally found in \cite{Ramirez14}.
\end{example}


\subsection{Monotonicity}\label{sec:Monotonicity}

Here we will generalize the results of \Lem{restrict} to show that the ellipses in the pencil are
sequentially nested if the parameter vectors are properly ordered. The following two lemmas will imply \Th{Monotonicity},
mentioned in the introduction.

\begin{lem}\label{lem:interior}
Assume that $C_1, C_2$ satisfy \Ass{SA1}-\ref{ass:SA3} and that $\bnu$ and $\bmu$ satisfy \Eq{nuRestrictions}.
Then, for the corresponding ellipses $\cC_\bmu ,\cC_\bnu$,
\[
	\cC_\bmu \subset \interior (\cC_\bnu)  \quad \Longleftrightarrow \quad \nu_1\mu_2-\nu_2\mu_1 > 0 ,
\]
i.e., the cross product $\bnu \times \bmu > 0$, for the vectors sketched in \Fig{PencilRegion}.
\end{lem}

\begin{proof}
First, we assume that $\nu_1\mu_2-\nu_2\mu_1 > 0$.
Let $(x,y)\in \cC_\bmu$ and $\bu=(x,y,1)$.
By assumption, $\bu^T C_\bmu \bu=0$. This implies that $\bu^T C_1 \bu<0$
because $\cC_\bmu \subset \interior (\cC_1)$. We will show that $\bu^T C_\bmu \bu<0$.
A simple computation shows that
\[
\mu_2\left(\bu^T C_\bnu \bu\right)-\nu_2\left(\bu^T C_\bmu \bu\right)=(\nu_1\mu_2-\nu_2\mu_1)\bu^T C_1 \bu.
\]
By assumption, $\mu_2>0$ and hence $\bu^T C_\bnu \bu<0$.  A similar computation shows the converse.
\end{proof}

As we noted in \Fig{PencilRegion}, the rotation number is a monotone increasing function along
lines in the $\bnu$ plane that start on the $\nu_1$-axis and end on the ray $\nu_1+\lambda_3\nu_2 = 0$.
This result follows easily from \Eq{PencilRotation}:

\begin{lem}\label{lem:monorotation}
  Let $C_1, C_2$  and $\bnu,\bmu$ be defined as above. Then
  $\rho(\bmu) > \rho(\bnu)$ if and only if $\bnu \times \bmu > 0$.
\end{lem}

\begin{proof}
Assume that $\bnu \times \bmu > 0$.
Since the elliptic integral $F(\omega,e)$ in \Eq{PencilRotation}
is a monotone increasing function of $\omega$ and $\arcsin$ is a
monotone increasing function on the interval $[0,1]$, it is sufficient to
show that its argument in \Eq{PencilRotation} is also an increasing function.
Note that
\[
	 \frac{\mu_2}{\lambda_1 \mu_2+\mu_1} -\frac{\nu_2}{\lambda_1 \nu_2+\nu_1}
	 = \frac{\nu_1\mu_2-\nu_2\mu_1 }{(\lambda _1 \nu _2+\nu _1)(\lambda _1 \mu _2+\mu _1)} >0,
\]
since the denominators above are positive by \Eq{nuRestrictions} and
$\nu_1\mu_2-\nu_2\mu_1 >0$ by assumption. This implies that
the argument of the $\arcsin$ in \Eq{PencilRotation} is monotone increasing.
Thus $\omega(\bmu)>\omega(\bnu)$, and hence  $\rho(\bmu) > \rho(\bnu)$.
The converse is shown in a similar way.
\end{proof}

\begin{cor}
  Two valid parameters $\bmu,\bnu$ satisfy $\rho(\bmu) =\rho(\bnu)$
  if and only if they are parallel.
\end{cor}

Together \Lem{interior} and \Lem{monorotation} are equivalent to \Th{Monotonicity} in \Sec{Introduction}.

\section{Full Poncelet Theorem}\label{sec:GeneralPoncelet}

As usual, let $\cC_{\bnu}$ the pencil generated by $\cC_1$ and $\cC_2$, so that the parameters
$\bnu$ satisfy \Eq{nuRestrictions}. In this section we consider compositions of
Poncelet maps for different elements of the pencil.

\begin{thm}\label{thm:Commuting}
 Let $P_{\bnu}$, $P_{\bmu}$ be Poncelet maps for valid parameters $\bnu,\bmu$, \Eq{nuRestrictions}.
Then
\begin{enumerate}
 \item $P_{\bnu}\circ P_{\bmu}$ has a lift using the covering $\Pi_M$ of the form $\theta\mapsto \theta+\rho(\bmu)+\rho(\bnu)$;
 \item the maps commute: $P_{\bnu}\circ P_{\bmu}=P_{\bmu}\circ P_{\bnu}$.
\end{enumerate}
\end{thm}

\begin{proof}
The result follows from \Lem{PencilPoncelet}.
Letting $\Pi_M$ \Eq{stdcover} be the standard covering map, then for each $\theta\in\bR$,
\[
	P_{\bnu}\left(P_{\bmu}(\Pi_M(\theta))\right)=
	P_{\bnu}\left(\Pi_M(\theta+\rho(\bmu))\right)=
	\Pi_M( \theta+\rho(\bmu)+\rho(\bnu)).
\]
Clearly, this also implies that $P_{\bnu}, P_{\bmu}$ commute.
\end{proof}

\begin{remark}
Note that \Th{Commuting} implies \Th{GeneralPoncelet} in \Sec{IntroStdCovering} directly.
\end{remark}

\begin{example}\label{ex:DualPoncelet}
Let $\overline{\bnu}$ be the dual of any valid parameter $\bnu$, recall \Th{DualParam}.
From \Th{Commuting} the composition $P_{\overline{\bnu}}\circ P_\bnu$
is a circle map that has rotation number $\tfrac12$. For each pencil, there is only one
Poncelet map with that rotation number and it has the limiting point $\cC_\infty$ as caustic.
Denote this map as $P_\infty=P_{\overline{\bnu}}\circ P_\bnu$.
Since this map has rotation number is $\tfrac12$ the composition
\[
  	P_\infty\circ P_\infty = P_{\overline{\bnu}}\circ P_\bnu\circ P_{\overline{\bnu}}\circ P_\bnu = id.
\]
This implies that the inverse of $P_\bnu$ can be written as
\[
	\left( P_\bnu\right)^{-1}=P_\infty\circ P_{\overline{\bnu}}=P_{\overline{\bnu}}\circ P_\infty.
\]
An example is sketched in \Fig{DualPoncelet}.
\end{example}

\InsertFig{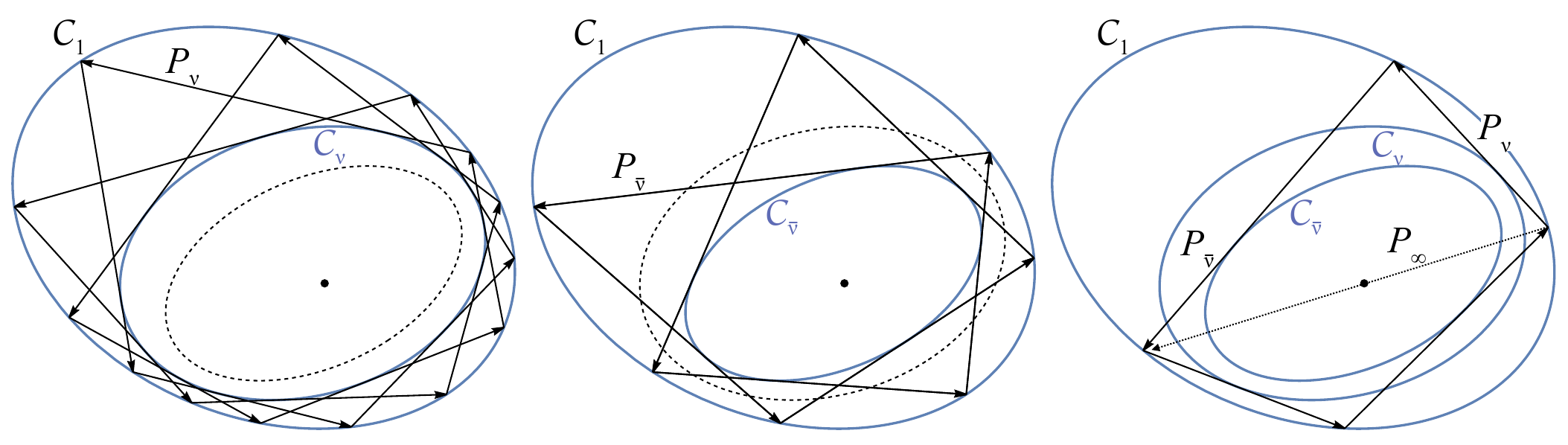}{Composition of a Poncelet Map and its dual as in \Ex{DualPoncelet}. The left
panel shows a map $P_\bnu$ with $\rho(\bnu) = \tfrac{3}{14}$. The middle panel shows
the map $P_{\overline{\bnu}}$ with $\rho(\overline{\bnu}) = \tfrac12 - \rho(\bnu) = \tfrac27$.
The composition $P_\infty=P_{\overline{\bnu}}\circ P_\bnu$, shown in the right panel as the dashed line,
goes through the limiting ellipse (the point in the figure) and thus has
rotation number $\tfrac12$. The pencil for this example has eigenvalues $\left\{\frac{1}{5},\frac{1}{8},\frac{1}{9}\right\}
$, and eccentricity $e = \sqrt{\frac{27}{32}}$ as in \Fig{Conjugacy}, however a different
map $\widehat{M}$ \Eq{mhat} was used.}{DualPoncelet}{7in}

As a corollary of these results we have obtained the general theorem of Poncelet.

\begin{cor}[Poncelet]\label{cor:Poncelet}
Suppose that $z_0,z_1\ldots,z_N\in \cC_1 $. Then there
exist a sequence of valid parameters $\bnu^1,\bnu^2,\ldots,\bnu^N$  and signs
$\sigma_1,\sigma_2,\ldots,\sigma_N\in\{1,-1\}$ such that,
 \begin{enumerate}
     \item For each $k=1,\ldots,N$ the segment $\overrightarrow{z_{k-1}z_{k}\,}$ is tangent
     to a unique ellipse $\cC_{\bnu^k}$. The orientation is positive if $\sigma_k=1$ and negative if $\sigma_k=-1$.

     \item The Poncelet maps $P_{\bnu^k}:\cC_1 \to \cC_1 $
     satisfy 
     $ \left(P_{\bnu^k}\right)^{\sigma_k}(z_{k-1})=z_{k}$.
     \item If $z_0=z_N$, then $P_* \equiv (P_{\bnu^N})^{\sigma_N}\circ (P_{\bnu^{N-1}})^{\sigma_{N-1}}\circ \cdots \circ(P_{\bnu^1})^{\sigma_1}=\id$.
 \end{enumerate}
\end{cor}

\begin{proof}
Using the standard covering \Eq{stdcover}, we can represent the points $z_0,\ldots,z_N$ as
\[
	z_k=\Pi_M(\theta_k),
\]
for suitable values of $\theta_k\in\bR$. Without loss of generality, we can assume that
or each $k=1,\ldots,N$,
\[
	0< \ell_k \equiv \left|\theta_{k}-\theta_{k-1}\right|\leq \tfrac{1}{2}, 
\]
and define $\sigma_k\in\{-1,1\}$ by
\[
	\theta_{k}-\theta_{k-1}=\sigma_k\,\ell_k.
\]
Using \Eq{ParamCanonical}, we then have
\[
	\bnu^k= \big( \nu_1(\ell_{k}), \nu_2(\ell_{k})\big)
		=\left(\lambda_1\cn^2(2K(e)\ell_{k},e)-\lambda_3,\,\sn^2(2K(e)\ell_{k},e)\right).
\]
By \Th{PencilParam}, $P_{\bnu^k}$ has rotation number $\rho(\bnu^k)=\ell_k$. In fact,
\[
	\left(P_{\bnu^k}\right)^{\sigma_k}(z_{k-1})=\Pi_M(\theta_{k-1}+\sigma_k\,\ell_k)=\Pi_M(\theta_{k})=z_{k}.
\]

Recalling \Fig{PonceletMap}, for each $z\in\cC_1$, the segment $\overrightarrow{zP_\bnu(z)}$
is tangent to $\cC_{\bnu}$ with positive  orientation,
and the segment $\overrightarrow{zP_\bnu^{-1}(z)}$ is tangent to $\cC_{\bnu}$ with negative orientation.
This implies that the segment  
$\overrightarrow{z_{k-1}z_{k}\,}$ is tangent to  $\cC_{\bnu^k}$, for each $k$, with positive/negative
orientation for $\sigma_k= \pm 1$.

Finally, we notice that
\[
	\theta_N-\theta_0=\sum_{k=1}^{N}\left( \theta_{k}-\theta_{k-1}\right)=\sum_{k=1}^{N}\sigma_k\,\ell_k.
\]
Consequently, if $z_0=z_N$, then $\theta_N-\theta_0\in\bZ$, and since
$\sum_{k=1}^{N}\sigma_k\,\ell_k$ is the rotation number of the composition, we have
that the rotation number of $P_*$ is an integer. Consequently
$P_*$ is the identity.
\end{proof}

\begin{remark}
The symmetry of the Poncelet map can be applied to any finite sequence $z_0,\ldots, z_N$, to give a
one-parameter family of orbits that will have the same caustics and orientations as the original.
Using the notation of the previous proof, define $\Theta_k(t)=\theta_k+t$ and
\beq{Zkt}
 	Z_k(t)=\Pi_M(\Theta_k(t)) = \Phi_t(z_k),
\eeq
where $\Phi_t:\cC_1 \to\cC_1 $ is the equivariant symmetry of \Cor{EquivariantSym}.

Note that, for each $t$ and each $k$,
$
	\Theta_{k}(t)-\Theta_{k-1}(t)=\sigma_k\,\ell_k.
$
Thus \Cor{Poncelet} implies that for each $t$ the segment $\overrightarrow{Z_{k-1}(t) Z_{k}(t)\,}$
is tangent to $\cC_{\bnu^k}$,
with the same orientation as $\overrightarrow{z_{k-1}z_{k}}$. In addition,
\[
	\left(P_{\bnu^k}\right)^{\sigma_{k}}(Z_{k-1}(t))=Z_{k}(t).
\]
\end{remark}

\begin{example}\label{ex:Composition}
To illustrate how the equivariant symmetry interacts with Poncelet theorem, \Fig{Composition} shows a composition of
Poncelet maps for a pencil with outer ellipse $\cC_1$ and a pair of inner ellipses
$\cC_{\bnu^1}$ and  $\cC_{\bnu^2}$ chosen so that the rotation numbers are
$\rho(\bnu^1) = \ell_1= \tfrac15$ and $\rho(\bnu^2) = \ell_2 = \tfrac{3}{10}$.
In the left panel of \Fig{Composition} a point $z_0$ is iterated
to obtain the images
\[
	z_1=P_{\bnu^1}(z_0),\quad z_2=(P_{\bnu^2})^{-1}(z_1),\quad
	z_3=P_{\bnu^1}(z_2),\quad z_4=P_{\bnu^1}(z_3),\quad
	z_5=z_0=(P_{\bnu^2})^{-1}(z_4).
\]
Using the notation of the previous corollary, we have $\bnu^1=\bnu^3=\bnu^4$,
$\bnu^2=\bnu^5$, $\sigma_1=\sigma_3=\sigma_4=1$, and
$\sigma_2=\sigma_5=-1$.
The center and right panels of \Fig{Composition} show the shape of the orbit with $Z_k(t)$ given by \Eq{Zkt} for two values of $t$.
\end{example}

\InsertFig{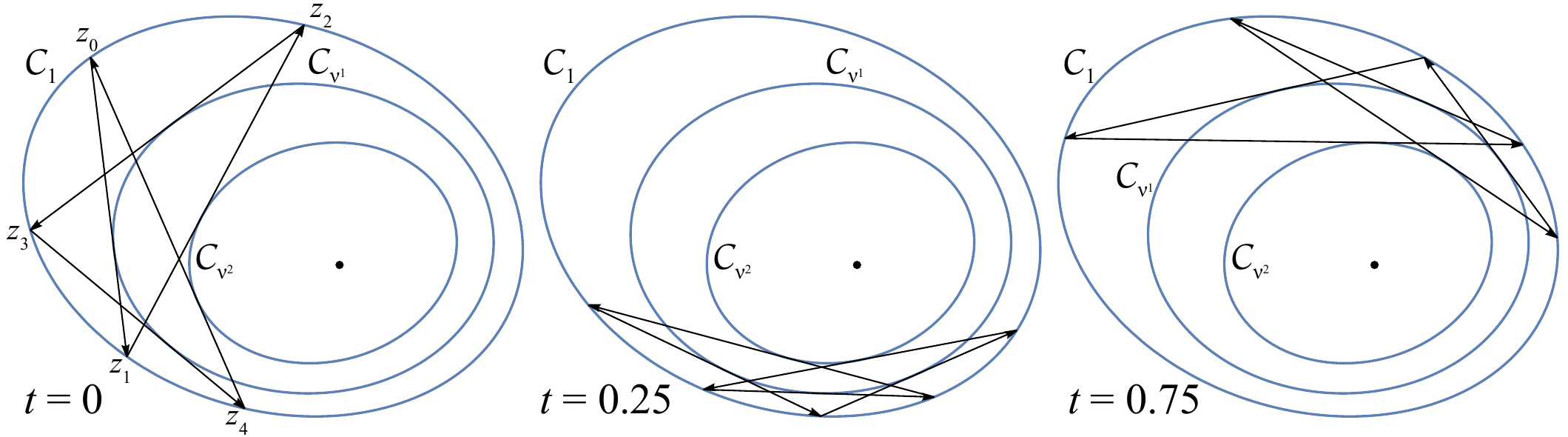}{Composition of a pair of Poncelet maps $P_{\bnu^1}$ and $P_{\bnu^2}$with rotation numbers
$\tfrac15$ and $\tfrac{3}{10}$, as in \Ex{Composition}. The resulting map has period $5$, so $P_* = id$. The center and
right panels show the translations of the orbit for $t=0.25$ and $t=0.75$ using the symmetry \Eq{Zkt}. } {Composition}{7in}

\begin{example}\label{ex:Cubic}
To illustrate \Cor{Poncelet}, we choose rotation numbers that are in the real cubic field $\mathbb{Q}(x)$ generated by
the single real root of the irreducible polynomial
\[
	21 x^3 + 14 x^2 +7 x - 8.
\]
The numbers
\bsplit{CubicEll}
	\ell_1 &= 3x^3 \approx 0.27910109553533, \\
	\ell_2 &= 2 x^2 \approx 0.41063584162374,\\
	\ell_3 &= x \approx 0.45312020569808
\esplit
generate the field, but are irrational and each pair is necessarily incommensurate:
there are no integers $(m,n,k)$ for which $m \ell_i + n \ell_j = k$, for $i,j \in  \{1,2,3\}$.
Choosing $\bnu^i$ using \Eq{ParamCanonical}, we obtain three Poncelet maps $P_{\bnu^i}$ with irrational rotation numbers.
For example, the first $90$ iterates of an orbit of $P_{\bnu^2}$ are shown in the left pane of  \Fig{CubicCombo},
and $78$ iterates of the alternating composition of $P_{\bnu^1}$ and $P_{\bnu^3}$ are shown in the middle pane.
Neither of these orbits will close, since the corresponding rotation numbers are irrational.
However, since $\ell_1+\ell_2+\ell_3 = \tfrac87$,
the composition $P_{\bnu^3}\circ P_{\bnu^2} \circ P_{\bnu^1}$ has rotation number $\tfrac87$. This map is shown
in the right panel of  \Fig{CubicCombo}. Defining $z_k$ as in \Cor{Poncelet}, we see that $z_{21} = z_0$ for any
starting point.
\end{example}

\InsertFig{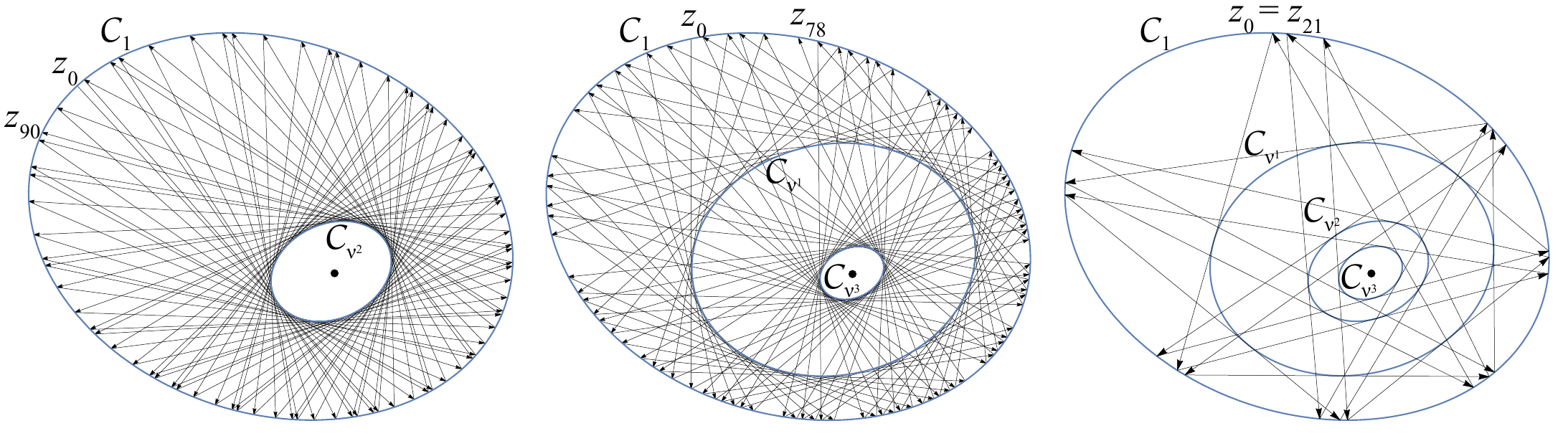}{Poncelet maps with rotation numbers in a cubic field as discussed in \Ex{Cubic}.
(Left) The map $P_{\bnu^2}$ with rotation number $\ell_2$ \Eq{CubicEll}.
(Middle) The map $P_{\bnu^3}\circ P_{\bnu^1}$ with rotation number $\ell_3+\ell_1$.
Since the rotation numbers are incommensurate, there is no combination of only two of the maps,
$P_{\bnu^1}$, $P_{\bnu^2}$, $P_{\bnu^3}$, that will be poristic.
(Right) The map $P_{\bnu^3}\circ P_{\bnu^2} \circ P_{\bnu^1}$ with rotation number $\tfrac87$.
Here the pencil is the same as that in \Fig{DualPoncelet}.}
{CubicCombo}{6.8in}


\section{Conclusions}
In this paper we have studied Poncelet's theorem from a dynamical systems point of view.
In \Sec{BilliardMaps} we analyzed elliptic billiard maps and found a formula for the rotation number.
We began by recalling the known symmetry of the elliptic billiard that implies its integrability.
This symmetry gives a  conjugacy for each orbit with an inner elliptical caustic to a map that
is a rigid translation on a covering space.
This resulted in an explicit expression for
the rotation number of the circle maps corresponding to Poncelet dynamics for the case of
confocal ellipses. The rotation number was written in terms of Jacobi elliptic functions
with modulus equal to the eccentricity of the inner ellipse, recall \Eq{BilliardRotationNumber}.

We showed in \Sec{PonceletMaps} that there is a conjugacy between the confocal case and the general Poncelet map for
elements of a pencil of ellipses. As we noted in \Sec{EllipticPencils},
the fundamental parameters of such a pencil are its three eigenvalues.
In particular, these determine the eccentricity of the pencil, \Eq{PencilEccentricity},
which again appears as the modulus of the elliptic functions that give the rotation number
for any element of the pencil, \Eq{PencilRotation}. In \Th{PencilParam} we obtained an explicit form for the choice of pencil element that gives a Poncelet map with a given rotation number.

When the rotation number is rational, all orbits of the Poncelet map are periodic. These continuous families of polygons are known as Poncelet porisms. Using the explicit form of the rotation number, we reproduced well-known explicit expressions for some of the low period porisms that arise from the classical conditions of Cayley, both for the confocal case in \Sec{Poristic} and for general pencils in \Sec{ExplicitRotation}.
We found that the conditions for porisms can be written solely in terms of
the eigenvalues of the pencil.

Finally, in \Sec{GeneralPoncelet} we showed how the general Poncelet theorem, for maps
with an arbitrary sequence of inner ellipses chosen from a pencil,
follows from the fact that a single covering space simplifies all the Poncelet maps in a pencil.
We gave an explicit expression \Eq{stdcover} for the covering space in terms of the projective
map \Eq{mhat} and the pencil eccentricity \Eq{PencilEccentricity}.

\newpage
\appendix
\begin{center}
\Large{\textbf{Appendices}}
\end{center}

\section{Confocal Poristic Examples}\label{app:Porisms}
In this appendix we recall two classical formulas for Poristic ellipses.
We consider a pair of confocal ellipses $E(e)$ and $E(f)$ with eccentricities $e$ and $f$ so that $0 < f < e < 1$.

\subsection{Rotation number $\tfrac13$}

We can show that the rotation number is $\tfrac13$ when the eccentricities are in the set \Eq{OneThird}.
To find such a pair of ellipses we should find an orbit whose segments form a triangle.
Consider the three points
\begin{align*}
	z_0&= \tfrac{1}{e} \left(1,-\tfrac{1}{f}\sqrt{(1-f^2)(e^2-f^2)}\right),\\
 	z_1&=\tfrac{1}{e} \left(1,\tfrac{1}{f}\sqrt{(1-f^2)(e^2-f^2)}\right),\\
	z_2&=\left(-\tfrac{1}{f},0\right),
\end{align*}
on $E(f)$. Clearly, the directed segment $\overrightarrow{z_0 z_1}$ is tangent
to the $E(e)$ at $(\tfrac{1}{e},0)$. To produce a triangle, we need to adjust $e$ so that
both segments $\overrightarrow{z_1 z_2}$ and $\overrightarrow{z_2 z_0}$
are tangent to $E(e)$.
To solve this problem, substitute the segment $(1-t)z_1 + tz_2$ into $E(e)$ to give the function
\[
	h(t)=e^2\big((1-t)x_1+tx_2\big)^2+\left(\frac{e^2}{1-e^2}\right)\big((1-t)y_1+ty_2\big)^2-1.
\]
The segment $\overrightarrow{z_1 z_2}$ is tangent to $E(e)$ if and only if $h$ has a double root in the interval
$(0,1)$, so that the discriminant of $h$ is necessarily zero---this is exactly \Eq{OneThird}.
We conclude that the segment $\overrightarrow{z_1 z_2}$ is tangent to $E(e)$ if and only if \Eq{OneThird}
is satisfied. By symmetry, $\overrightarrow{z_2 z_0}$ is also tangent to $E(e)$. This implies that
$\{z_0,z_1,z_2\}$ is an orbit of $B_e^f$ and hence $\rhoin(e,f)=1/3$.

\subsection{Cayley's Solution for Period$-5$ orbits}
As mentioned in \Sec{IntroPorism}, the zero level sets of the polynomial $\cay_N$
correspond to Poncelet porisms with period $N$ \cite{Griffiths78, Dragovic98a, Dragovic98b}.
If $N=5$, the result allows us to compute a polynomial $\cay_5(e,f)$
such that
\[
	\cay_5^{-1}\{0\}\cap \Delta=	\poris{1/5}\cup \poris{2/5}.
\]
If the sets $\poris{1/5}$ and $\poris{2/5}$ can be written as zeros of polynomials,
then they have to be factors of $\cay_5$. Using symbolic manipulation, we found
\begin{align*}
	\poris{1/5} =\{(e,f)\in\Delta&: e^6+8 e^5 f^5-10 e^5 f^3+2 e^5 f-4 e^4 f^6+5 e^4 f^4\\
				&-4 e^4 f^2-12 e^3 f^7+12 e^3 f^5+4 e^2 f^{10}-5 e^2 f^8\\
 				&+4 e^2 f^6-2 e f^{11}+10 e f^9-8 e f^7-f^{12}=0\}, \\
	\poris{2/5}=\{(e,f)\in\Delta&: e^6-8 e^5 f^5+10 e^5 f^3-2 e^5 f-4 e^4 f^6+5 e^4 f^4\\
				&-4 e^4 f^2+12 e^3 f^7-12 e^3 f^5+4 e^2 f^{10}-5 e^2 f^8\\
				&+4 e^2 f^6+2 e f^{11}-10 e f^9+8 e f^7-f^{12}=0\}.
\end{align*}
Using formula \Eq{poristic},
we can verify numerically that $\poris{1/5}$ and $\poris{2/5}$ are the poristic sets.

\section{Asymptotic Limits}\label{app:SingularLimit}
For each $0<f<1$ we want to show that
\[
	\lim_{e\to 1^-}\rhoin(e,f)=\tfrac{1}{2},
\]
as mentioned in \Rem{limits}.
By \Th{BilliardRotationNumber}, for each $f<e<1$,
\[
	0 < \frac{1}{2}-\rhoin(e,f)=\frac{K(e)-F(\omegain(e,f),e)}{2K(e)}
		=\frac{1}{2K(e)}\int_{\omegain(e,f)}^{\pi/2}\frac{d\tau}{\sqrt{1-e^2\sin^2\tau}},
\]
where $\omegain(e,f)$ is given in \Eq{BilliardRotationNumber}.
The denominator in the integral above is monotone decreasing on
$[0,\pi/2]$, with a minimum value of $\sqrt{1-e^2}$. Therefore,
\[
	\left|\frac{1}{2}-\rhoin(e,f)\right|\leq
		\frac{1}{2K(e)}\left(\frac{\pi/2-\omegain(e,f)}{\sqrt{1-e^2}}\right).
\]
A simple computation then gives
\[
	\lim_{e\to 1^-}\frac{\pi/2-\omegain(e,f)}{\sqrt{1-e^2}}=\frac{\sqrt{1-f^2}}{f}<\infty.
\]
Therefore, since $K(e)\to\infty$, as $e\to1^-$, we conclude that
\(\displaystyle
	\lim_{e\to 1^-}\rhoin(e,f)=\tfrac{1}{2}.
\)

In \Rem{DeltaFoliation}, we mentioned the following limit
\[
\lim_{e\to 1^-}e\,\cd(2K(e)\ell,e)=1.
 \]
This can be shown, using results in \cite[Thm. 3]{Carlson83} and \cite[Lem. 2]{Siegel2016}.
In fact, fixing $\ell\in(0,1/2)$, we have the following asymptotic expansions,
as $e\to 1^-$.
\[
	\cn(2K(e)\ell,e) =2^{1-\ell}(1-e^2)^\ell+o\left((1-e^2)^\ell\right),
\]
\[
	\dn(2K(e)\ell,e)= 2^{1-\ell}(1-e^2)^\ell+o\left((1-e^2)^\ell\right).
\]
These imply the limit above.

\section{Derivatives of the billiard rotation number}\label{app:Derivatives}

To show that the rotation number is a monotone function of $e$ and $f$, we first
obtain some related results for elliptic functions.
The \emph{Jacobi epsilon function} is defined as
\[
	\cE(u,e)=\int_{0}^{u}\dn^2(\tau,e) \,d\tau = E(\am(u,e),e),
\]
where $E$ is the incomplete elliptic integral of the second kind,
\[
	E(\phi,e)=\int_{0}^{\phi}\sqrt{1-e^2\sin^2\tau}\,d\tau .
\]
In particular, the complete elliptic integral of the second kind becomes
$\cE(K(e),e)=E(\pi/2,e)$.
We first prove the following auxiliary result. 
\begin{lem} \label{lem:EpsilonInequality}
Let $z,e\in(0,1)$ and define
 \(
	g(z) \equiv \cE( z\,K(e),e)-z\,\cE( K(e),e).
\)
Then $g(z)>0$.
\end{lem}

\begin{proof}
The second derivative of $g$ is
\[
	 g''(z)=-2e^2 K^2\left(e\right)\dn \left(z K\left(e\right),e\right)
	 	\cn \left(z K\left(e\right),e\right) \sn \left(z K\left(e\right),e\right).
\]
Since $z\in(0,1)$ then $g''(z)<0$, so $g$ is strictly concave on $(0,1)$.
Moreover, given that $g(0)=g(1)=0$, then $g(z)>0$, for all $z\in(0,1)$.
\end{proof}

To prove the next lemma, we will use derivatives of the elliptic functions and integrals with respect to
the argument and modulus, as found in \cite[Formulas 710.00, 710.61, 731.11]{Byrd1971}:
\bsplit{BasicElliptic}
	\frac{\partial}{\partial e}K(e) &= \frac{1}{e\,(1-e^2)}\left[\cE(K(e),e)-(1-e^2)K(e)\right],\\
	\frac{\partial}{\partial e}\cd(t,e)&= \frac{\,\sn(t,e)}{e\,\dn^2(t,e)}\left[\cE(t,e)-(1-e^2)t\right],\\
	\frac{\partial}{\partial t}\cd(t,e)&= - (1-e^2) \frac{\sn(t,e)}{\dn^2(t,e)}.
\esplit

\begin{lem}\label{lem:paramgrowth} Defining $f(\ell,e)$ as in \Eq{poristic},
then $\displaystyle\frac{\partial f}{\partial \ell}<0$ and $\displaystyle\frac{\partial f}{\partial e}>0$.
\end{lem}

\begin{proof}
To clarify the notation let $R:(0,1/2)\times(0,1)\to \bR$ be the function
\beq{FunctionR}
 	R(\ell,e)= e\,\cd(2 K(e)\ell,e),
\eeq
given in \Eq{poristic}. Using \Eq{BasicElliptic}, we find
\[
 	\frac{\partial R}{\partial \ell}=-2e(1-e^2) \frac{\sn(2 K(e)\ell,e)}{\dn^2(2 K(e)\ell,e)}\,K(e),
 \]
 \[
 \frac{\partial R}{\partial e}=\cd(2 K(e)\ell,e)+
 \frac{\,\sn(2 K(e) \ell,e)}{\dn^2(2 K(e) \ell,e)}\big[\cE(2 K(e) \ell,e)-2 \ell\,\cE( K(e),e)\big].
 \]
 If $\ell\in(0,1/2)$, then $\cd$, $\sn$, and $\dn$ are positive with the arguments above.
 Using \Lem{EpsilonInequality}, we find that $\cE(2 K(e) \ell,e)-2 \ell\,\cE( K(e),e)>0$.
 Since \Eq{poristic} implies that $f = R(\ell,e)$, we conclude that $\displaystyle\frac{\partial f }{\partial \ell}<0$ and $\displaystyle\frac{\partial f}{\partial e}>0$.
\end{proof}

\begin{lem}\label{lem:rotgrowth}
 The billiard rotation number $\rhoin(e,f)$ given by \Eq{BilliardRotationNumber} satisfies
 $\displaystyle\frac{\partial \rhoin}{\partial e}>0$ and $\displaystyle\frac{\partial \rhoin}{\partial f}<0$.
\end{lem}

\begin{proof}
 Using \Eq{FunctionR} together with \Eq{poristic} gives the implicit expression
 \[
	R(\rhoin(e,f),e)=f.
 \]
 This implies that
\[
 \frac{\partial R}{\partial \ell}\cdot\frac{\partial \rhoin}{\partial e}+\frac{\partial R}{\partial e}=0, \qquad
 \frac{\partial R}{\partial \ell}\cdot\frac{\partial \rhoin}{\partial f}=1.
 \]
Since $\displaystyle\frac{\partial R}{\partial \ell}<0$ and $\displaystyle\frac{\partial R}{\partial e}>0$,
this implies the result.
\end{proof}

\bibliography{biblio}{}
\bibliographystyle{alpha}
\end{document}